\documentclass[11pt]{amsart}

\usepackage{amsmath,amssymb,amsthm,mathtools}
\usepackage{mathrsfs}
\usepackage[margin=1.1in]{geometry}
\usepackage{enumitem}
\usepackage{xcolor}
\usepackage{tikz}
\usepackage{float}
\usepackage[hidelinks]{hyperref}
\DeclareRobustCommand{\finalrev}[1]{\begingroup#1\endgroup}
\DeclareRobustCommand{\rev}[1]{#1}
\newenvironment{revision}{}{}

\numberwithin{equation}{section}
\newtheorem{theorem}{Theorem}[section]
\newtheorem{proposition}[theorem]{Proposition}
\newtheorem{lemma}[theorem]{Lemma}
\newtheorem{corollary}[theorem]{Corollary}
\theoremstyle{definition}
\newtheorem{definition}[theorem]{Definition}
\newtheorem{example}[theorem]{Example}
\theoremstyle{remark}
\newtheorem{remark}[theorem]{Remark}

\newcommand{\kk}{\Bbbk}
\newcommand{\kb}{\overline{\kk}}
\newcommand{\Nil}{\mathcal{N}}
\newcommand{\cO}{\mathcal{O}}
\newcommand{\sO}{\mathscr{O}} % rational wavefront sets
\newcommand{\cS}{S} % rational-orbit coordinate map
\newcommand{\cF}{\mathcal{F}}
\newcommand{\tF}{\widetilde{\mathcal{F}}}
\newcommand{\st}{\mathrm{st}}
\newcommand{\mx}{\mathrm{max}}
\newcommand{\I}{\mathrm{I}}
\newcommand{\II}{\mathrm{II}}
\newcommand{\g}{\mathfrak{g}}
\newcommand{\so}{\mathfrak{so}}
\newcommand{\spl}{\mathfrak{sp}}
\newcommand{\sltwo}{\mathfrak{sl}_2}
\newcommand{\bA}{\overline{A}}
\newcommand{\bS}{\overline{\cS}}
\DeclareMathOperator{\Wh}{Wh}
\DeclareMathOperator{\Res}{Res}
\DeclareMathOperator{\Ind}{Ind}
\DeclareMathOperator{\Ad}{Ad}
\DeclareMathOperator{\ad}{ad}
\DeclareMathOperator{\pr}{pr}
\DeclareMathOperator{\disc}{disc}
\DeclareMathOperator{\Irr}{Irr}
\DeclareMathOperator{\Hom}{Hom}
\DeclareMathOperator{\Cay}{Cay}
\DeclareMathOperator{\sgn}{sgn}
\DeclareMathOperator{\Lie}{Lie}
\newcommand{\Lcan}{\mathcal{L}^{\mathrm{can}}}
\newcommand{\ResG}{\Res^{G^F}_{G_0^F}}
\newcommand{\IndG}{\Ind^{G^F}_{G_0^F}}

\begin{document}

\title{A note on rational wavefront sets for finite special orthogonal groups}

\author{Pengcheng Li}
\address{Department of Mathematics, The University of Hong Kong,
Hong Kong, China}
\email{pcli25@hku.hk}

\author{Zhicheng Wang}
\address{Department of Mathematics, Jilin University,
Changchun 130012, Jilin, P.~R. China}
\email{wangzhicheng@jlu.edu.cn}

\hypersetup{pdfauthor={Pengcheng Li and Zhicheng Wang}}

\date{}

\begin{abstract}
We study rational wavefront sets of irreducible representations of finite
special orthogonal groups over $\mathbb F_q$. We assume that $q$ is odd and
that both $q$ and its characteristic are sufficiently large. Using {the second author's} results for full orthogonal
groups, we determine the maximal rational wavefront sets and the dimensions of
the corresponding generalized Whittaker models. We show that each rational
wavefront set is the rational closure of its maximal elements. For unipotent
representations, we describe the maximal rational wavefront sets in terms of
Lusztig's canonical quotient and their behavior under Alvis--Curtis duality.
\end{abstract}

\maketitle

% ===============================================================
\section{Introduction}
% ===============================================================

Let $\kk=\mathbb F_q$ be a finite field of odd characteristic $p$, and let
$\kb$ be an algebraic closure of $\kk$. Let $(V,B)$ be a nondegenerate
symmetric space  over $\kk$.
Put $V_{\kb}=V\otimes_{\kk}\kb$,
$G=\mathrm O(V_{\kb})$ and $G_0=\mathrm{SO}(V_{\kb})$, and let $F$ be the
Frobenius endomorphism induced by the $\kk$-structure of $V$. Thus
$G^F=\mathrm O(V)$ and $G_0^F=\mathrm{SO}(V)$.

Write $\g=\Lie(G)=\Lie(G_0)$ for the common Lie algebra, so that
$\g^F=\so(V)$, and let $\Nil\subset\g$ be the nilpotent cone. 

As in \cite[Section~1]{Wang}, identify
$\kk^\times/(\kk^\times)^2$ with $\{\pm\}$ and let $\varepsilon(-1)$ denote
the square class of $-1$. We use the discriminant convention
\[
\disc(V)=\varepsilon(-1)^{\dim V(\dim V-1)/2}\det(B)
\in\kk^\times/(\kk^\times)^2,
\]
where $\det(B)$ denotes the square class of the determinant of a Gram matrix of $B$.
With $\epsilon=\disc(V)$, we have
\[
G^F\cong\mathrm O_{\dim V}^{\epsilon}(\kk),\qquad
G_0^F\cong\mathrm{SO}_{\dim V}^{\epsilon}(\kk).
\]
The sign specifies the rational form; in even dimension, the sign $+$ gives
the split form. 

\finalrev{For $\dim V\geqslant3$, we impose the characteristic hypotheses
of \cite{Wang}; in particular $p>3(h_G-1)$, where $h_G$ is
the Coxeter number of $G_0$. We also assume that $q$ is
sufficiently large for \cite[Theorem~2.2]{Wang} to apply.}
Fix a nontrivial additive character $\psi\colon\kk\to\mathbb C^\times$.

\subsection{Rational orbits and wavefront sets}

For $J=G,G_0$ and $X\in\Nil^F$, set
\[
\cO_X^J=\Ad(J^F)X\subset\Nil^F,\qquad
(\cO_X^J)^{\st}=\Ad\bigl(J(\kb)\bigr)X\subset\Nil.
\]
The first is an $F$-rational nilpotent orbit, and the second is the $F$-stable orbit containing it, which we call its \emph{stable orbit}. We write $\Nil^F/\Ad(J^F)$ for the set of rational
orbits. \rev{The partition of an orbit in $\Nil$ is the Jordan type of its elements
acting on $V_{\kb}$. It is called \emph{very even} if all its parts are even. A
{stable} $G$-orbit is a single {stable} $G_0$-orbit unless its partition is
very even, in which case it is the union of two {stable} $G_0$-orbits; see
\cite[Section~5.1]{CM} and Lemma~\ref{lem:split}.}

For {stable} orbits of the same algebraic group, write
{$\cO_1^{\st}\leqslant_{\st}\cO_2^{\st}$ if}
\[
\cO_1^{\st}\subset\overline{\cO_2^{\st}},
\]
where the bar is Zariski closure in $\Nil$ over $\kb$.
{For rational $G_0^F$-orbits, we use the order
$\leqslant_{G_0}$ of Definition~\ref{def:order}; for $\leqslant_G$ on
rational $G^F$-orbits, see \cite[Section~3.2]{Wang}.}
Let $\sigma\in\Irr(G_0^F)$ and let $\cO_0$ be a rational $G_0^F$-orbit.
{Let $\Gamma^{G_0}_{\cO_0}$ be the generalized Gelfand--Graev
representation and let $\Wh_{\cO_0}(\sigma)$ be the generalized Whittaker
model (see \eqref{eq:GGGR} and \eqref{eq:Wh} for the definitions).}
The \emph{rational wavefront set} of $\sigma$ is
\[
\sO(\sigma):=\bigcup_{\substack{\cO_0\in\Nil^F/\Ad(G_0^F)\\ \Wh_{\cO_0}(\sigma)\neq0}}\cO_0 .
\]
Following \cite[Section~1]{Wang}, we identify a collection of rational
orbits with its union in $\Nil^F$. \rev{Thus $\cO_0\in\sO(\sigma)$ means that
$\cO_0$ is one of the orbits in this union, and for such a union $\mathcal U$
we write $|\mathcal U|$ for the number of its orbits.} Maps such as $\pr$ and
$\cS$ act on the corresponding collections of orbits; when their images and
inverse images are written as wavefront sets, they are interpreted as unions.
With this convention we put
\begin{align*}
\sO(\sigma)^{\mx}&:=\bigcup\bigl\{\cO_0\in\sO(\sigma):
  \text{there is no }\cO_1\in\sO(\sigma)\text{ with }\cO_0<_{G_0}\cO_1\bigr\},\\
\sO(\sigma)^{\mx,\st}&:=\bigcup\bigl\{\cO_0\in\sO(\sigma):
  \text{there is no }\cO_1\in\sO(\sigma)\text{ with }\cO_0^{\st}<_{\st}\cO_1^{\st}\bigr\}.
\end{align*}
For $\pi\in\Irr(G^F)$, the sets $\sO(\pi)$, $\sO(\pi)^{\mx}$ and
$\sO(\pi)^{\mx,\st}$ are defined in the same way, using $G^F$-orbits and the
order $\leqslant_G$.

\subsection{Background and motivation}

Kawanaka introduced generalized Gelfand--Graev representations to
study irreducible characters through nilpotent orbits \cite{Kaw}. By the work
of Lusztig, Geck--Malle, Achar--Aubert and Taylor \cite{Lus92,GM,AA,Tay}, for
every irreducible representation $\sigma$ of $G_0^F$ the set
\[
\bigl\{\cO_0^{\st}:\cO_0\in\sO(\sigma)^{\mx,\st}\bigr\}
\]
{consists of a single orbit, called the \emph{Kawanaka wavefront orbit}
of $\sigma$ and denoted by $\cO_\sigma^{\st}$.}
{With this notation,}
\[
\sO(\sigma)^{\mx,\st}
=\bigcup_{\substack{\cO_0\in\sO(\sigma)\\
                    \cO_0^{\st}=\cO_\sigma^{\st}}}\cO_0
\subset(\cO_\sigma^{\st})^F.
\]
\rev{Similarly, for $\pi\in\Irr(G^F)$ we write $\cO_\pi^{\st}$ for the
{stable} $G$-orbit containing $\sO(\pi)^{\mx,\st}$. In \cite{Wang} this
{stable} orbit is also denoted by $\sO(\pi)^{\mx,\st}$, by abuse of notation;
we distinguish between the two.}

For finite symplectic and full orthogonal groups, {the second author} proved
\cite[Theorem~1.1]{Wang} that
\[
\sO(\pi)^{\mx}=\sO(\pi)^{\mx,\st},\qquad
\sO(\pi)=\overline{\sO(\pi)^{\mx}},
\]
where the bar denotes closure with respect to the rational order. {The second author} also
determined $\sO(\pi)^{\mx}$ explicitly and computed the corresponding Whittaker
dimensions \cite[Theorems~1.4 and~1.9]{Wang}. For unipotent representations, he
described $\sO(\pi)^{\mx}$ in terms of fibers of Lusztig's canonical quotient and proved
that this description is compatible with Alvis--Curtis duality
\cite[Theorems~1.10 and~1.11]{Wang}.

Our aim is to deduce analogous statements for $G_0^F$. Since $G$ and $G_0$
have the same Lie algebra, their generalized Whittaker models are built from
the same unipotent subgroups and characters. A rational $G^F$-orbit may
\rev{split} into two rational $G_0^F$-orbits, and the restriction of an
irreducible representation of $G^F$ may split into two irreducible
representations of $G_0^F$. To describe
$\sO(\sigma)^{\mx}$, we must understand the relation between these two splittings.

\subsection{From full to special orthogonal groups}

Consider the surjection
\[
\pr\colon\Nil^F/\Ad(G_0^F)\longrightarrow\Nil^F/\Ad(G^F),\qquad
\Ad(G_0^F)X\longmapsto\Ad(G^F)X,
\]
so that $\pr(\cO_0)$ is the rational $G^F$-orbit containing $\cO_0$. {For} a rational $G^F$-orbit $\cO$ with partition
$\lambda$,
\[
\#\pr^{-1}(\{\cO\})=
\begin{cases}
1,&\lambda\text{ is not very even},\\
2,&\lambda\text{ is very even}.
\end{cases}
\]
If $\lambda$ is very even, then $V$ is split, there is exactly one rational
$G^F$-orbit $\cO_\lambda$ with partition $\lambda$, and {we write
$\pr^{-1}(\{\cO_\lambda\})=\{\cO_\lambda^{\I},\cO_\lambda^{\II}\}$,
where $\I$ and $\II$ label the two rational $G_0^F$-orbits contained in
$\cO_\lambda$. Then}
\[
\bigl((\cO_\lambda^{\I})^{\st}\bigr)^F=\cO_\lambda^{\I},\qquad
\bigl((\cO_\lambda^{\II})^{\st}\bigr)^F=\cO_\lambda^{\II}.
\]
Every $\sigma\in\Irr(G_0^F)$ occurs in the restriction of some
$\pi\in\Irr(G^F)$. By Clifford theory,
\begin{equation}\label{eq:clifford}
\begin{aligned}
\pi\otimes\det\not\simeq\pi&\implies \ResG\pi=\sigma\ \text{is irreducible},\\
\pi\otimes\det\simeq\pi&\implies \ResG\pi=\sigma^{\I}\oplus\sigma^{\II},
\end{aligned}
\end{equation}
and in the second case $\sigma^{\I}\not\simeq\sigma^{\II}$.
The second case occurs only when $\dim V$ is even.

\subsection{Maximal rational orbits}

\begin{theorem}\label{thm:A}
For every $\sigma\in\Irr(G_0^F)$,
\begin{equation}\label{eq:max-equals-maxst}
\sO(\sigma)^{\mx}=\sO(\sigma)^{\mx,\st}.
\end{equation}
Let $\pi\in\Irr(G^F)$ be such that $\sigma$ occurs in $\ResG\pi$. Then
\begin{equation}\label{eq:max-formula}
\sO(\sigma)^{\mx}
=\bigcup_{\substack{\cO_0\in\Nil^F/\Ad(G_0^F)\\
\pr(\cO_0)\in\sO(\pi)^{\mx},\ \cO_0^{\st}=\cO_{\sigma}^{\st}}}\cO_0.
\end{equation}
Let $\lambda$ be the partition of $\cO_{\pi}^{\st}$. If $\lambda$ is not
very even, then $\pr$ {induces} a bijection
$\sO(\sigma)^{\mx}\xrightarrow{\sim}\sO(\pi)^{\mx}$. If $\lambda$ is very even,
then $\ResG\pi$ is reducible and, with a suitable labeling of its constituents,
\begin{equation}\label{eq:very-even-separation}
\sO(\sigma^{\I})^{\mx}=\cO_\lambda^{\I},\qquad
\sO(\sigma^{\II})^{\mx}=\cO_\lambda^{\II}.
\end{equation}
\end{theorem}

\rev{Since $\sO(\pi)^{\mx}$ is determined explicitly in \cite{Wang},
formula \eqref{eq:max-formula} determines $\sO(\sigma)^{\mx}$ once the
Kawanaka wavefront orbit of $\sigma$ is known. In the very even case the
Kawanaka wavefront orbit decides which of $\cO_\lambda^{\I}$ and
$\cO_\lambda^{\II}$ lies in $\sO(\sigma)^{\mx}$, and a labeling satisfying
\eqref{eq:very-even-separation} exists by Proposition~\ref{prop:separation}.}

\begin{theorem}\label{thm:B}
For every $\sigma\in\Irr(G_0^F)$,
\[
\sO(\sigma)=\overline{\sO(\sigma)^{\mx}}.
\]
\end{theorem}

\subsection{Component groups and Whittaker multiplicities}

For $X\in\Nil$, write $\cO^{\st}=\Ad(G_0(\kb))X$ and
$\widetilde{\cO}^{\st}=\Ad(G(\kb))X$. Define the component groups by
\[
A_{\mathrm{SO}}(\cO^{\st}):=C_{G_0}(X)/C_{G_0}(X)^\circ,\qquad
A_{\mathrm{O}}(\widetilde{\cO}^{\st}):=C_G(X)/C_G(X)^\circ.
\]
When both groups are taken at the same $X$, we drop the tilde and write
$A_{\mathrm{O}}(\cO^{\st})$, as in \cite[Section~3.1]{Wang}. Since
$C_{G_0}(X)^\circ=C_G(X)^\circ$, the inclusion induces an injection
$A_{\mathrm{SO}}(\cO^{\st})\hookrightarrow A_{\mathrm{O}}(\cO^{\st})$. The
determinant is trivial on $C_G(X)^\circ$, so it induces a homomorphism
\[
\det\colon A_{\mathrm{O}}(\cO^{\st})\longrightarrow\{\pm1\},\qquad
gC_G(X)^\circ\longmapsto\det(g),
\]
and $A_{\mathrm{SO}}(\cO^{\st})=\ker(\det)$.
\begingroup
Let $\bA_{\mathrm O}(\cO^{\st})$ and
\[
q_{\mathrm O}\colon A_{\mathrm O}(\cO^{\st})\twoheadrightarrow
\bA_{\mathrm O}(\cO^{\st})
\]
be the quotient and quotient map specified in Section~\ref{subsec:coords}. For the partition $\lambda$ of $X$, we have
\[
\begin{cases}
[A_{\mathrm O}:A_{\mathrm{SO}}]=2,&\lambda\text{ is not very even},\\
A_{\mathrm O}=A_{\mathrm{SO}}=\bA_{\mathrm O}=1,&\lambda\text{ is very even}.
\end{cases}
\]
\endgroup

Let $\sigma\in\Irr(G_0^F)$, and choose $\pi\in\Irr(G^F)$ whose restriction
contains $\sigma$.
\begin{revision}%
Let $G^*$ be the dual group with Frobenius structure as in
\cite[Section~2.3]{Wang}, and let $s\in G^{*F}$ be the semisimple
parameter of $\pi$. Let
$\pi^*=\Lcan(\pi)$ be the unipotent representation of $C_{G^*}(s)^F$
attached to $\pi$ by the {Lusztig correspondence defined in}
\cite[Theorem~2.2]{Wang}. Put
\[
\cO^{\st}=\cO_{\sigma}^{\st},\qquad
\cO^{\st*}=\cO_{\pi^*}^{\st},
\]
where $\cO^{\st*}$ is the {stable} wavefront orbit of $\pi^*$ in
$\Lie C_{G^*}(s)$. In the following formula,
$\bA(\cO^{\st*})$ denotes the dual-side canonical quotient in
\cite[Theorem~1.9]{Wang}.
\end{revision}

\begin{theorem}\label{thm:C}
For every $\cO_0\in\sO(\sigma)^{\mx}$,
\begin{equation}\label{eq:Wh-dim-SO}
\dim\Wh_{\cO_0}(\sigma)=
\begin{cases}
\dfrac{2\,|A_{\mathrm{SO}}(\cO^{\st})|}
      {|\sO(\sigma)^{\mx}|\,|\bA(\cO^{\st*})|},
 & \dim V\text{ even and }\ResG\pi\text{ irreducible},\\[3ex]
\dfrac{|A_{\mathrm{SO}}(\cO^{\st})|}
      {|\sO(\sigma)^{\mx}|\,|\bA(\cO^{\st*})|},
 & \text{otherwise}.
\end{cases}
\end{equation}
This dimension is a power of $2$ and does not depend on
$\cO_0\in\sO(\sigma)^{\mx}$. If the partition of $\cO^{\st}$ is very even, the
dimension equals $1$.
\end{theorem}

For comparison, the formula for the full orthogonal group
\cite[Theorem~1.9]{Wang} reads, for $\cO\in\sO(\pi)^{\mx}$,
\begin{equation}\label{eq:Wh-dim-O}
\dim\Wh_{\cO}(\pi)=
\begin{cases}
\dfrac{|A_{\mathrm{O}}(\cO^{\st})|}
      {|\sO(\pi)^{\mx}|\,|\bA(\cO^{\st*})|},&\dim V\text{ even},\\[3ex]
\dfrac{|A_{\mathrm{O}}(\cO^{\st})|}
      {2\,|\sO(\pi)^{\mx}|\,|\bA(\cO^{\st*})|},&\dim V\text{ odd}.
\end{cases}
\end{equation}
\rev{The factor $2$ in the first case of \eqref{eq:Wh-dim-SO} compensates for
the index $[A_{\mathrm{O}}(\cO^{\st}):A_{\mathrm{SO}}(\cO^{\st})]=2$. Indeed, in
that case $\sigma=\ResG\pi$, so $\sigma$ and $\pi$ have the same Whittaker
dimensions; moreover, the partition of $\cO^{\st}$ is not very even by
Theorem~\ref{thm:A}, so the index is $2$.}

\subsection{Unipotent representations}

Let $\cO^{\st}$ be a stable nilpotent orbit of $G_0$. We use
the full orthogonal component group $A_{\mathrm{O}}(\cO^{\st})$ and canonical
quotient $\bA_{\mathrm{O}}(\cO^{\st})$ of \cite[Section~3.1]{Wang}. In
Section~\ref{subsec:coords} we define
\[
\cS\colon\{\cO_0\in\Nil^F/\Ad(G_0^F):\cO_0\subset(\cO^{\st})^F\}\longrightarrow
A_{\mathrm{O}}(\cO^{\st})\qquad\text{and}\quad \bS=q_{\mathrm{O}}\circ\cS.
\]
\rev{If the partition of $\cO^{\st}$ is very even, the domain of $\cS$ has a
single element, namely $(\cO^{\st})^F$.}
For $x\in\bA_{\mathrm{O}}(\cO^{\st})$ we set
\[
\sO(x)=\bigcup_{\substack{\cO_0\in\Nil^F/\Ad(G_0^F)\\
\cO_0\subset(\cO^{\st})^F,\ \bS(\cO_0)=x}}\cO_0.
\]

\begin{theorem}\label{thm:D}
Let $\sigma$ be an irreducible unipotent representation of $G_0^F$, and put
$\cO^{\st}=\cO_{\sigma}^{\st}$. Then
\[
\sO(\sigma)^{\mx}=\sO(x)\quad\text{for some }x\in\bA_{\mathrm{O}}(\cO^{\st}),
\]
and $\sO(\sigma)^{\mx}$ does not depend on $\psi$.
\rev{If $\disc(V)=+$ and $\sigma$ is the special member of its family, in the
sense of Section~\ref{subsec:families}, then}
\[
\cS\bigl(\sO(\sigma)^{\mx}\bigr)=\ker q_{\mathrm{O}}.
\]
\end{theorem}

For a family $\cF$ of unipotent representations of $G_0^F$, let
$\cO^{\st}_{\cF}$ denote its stable unipotent support, viewed as a nilpotent orbit \rev{by means of the Cayley map fixed in
Section~\ref{sec:models}}. Let $D_{G_0^F}$ be
Alvis--Curtis duality on virtual characters. For irreducible $\sigma$, let
$\sigma'$ be the irreducible representation such that
$D_{G_0^F}(\sigma)=\varepsilon_\sigma\sigma'$, $\varepsilon_\sigma\in\{\pm1\}$.
Choose a unipotent $\pi\in\Irr(G^F)$ whose restriction contains $\sigma$, let
$\tF$ be the family of $\pi$, and let $(x^+,x^-)$ be the parameter of $\pi$ as
in \cite[Section~2.6]{Wang}. \rev{Section~\ref{subsec:families} explains how these parameters pass to
$G_0^F$, and Section~\ref{subsec:AC} recalls the map $\Phi_{\tF}$ of
\cite[(9.52)]{Wang}.}

\begin{theorem}\label{thm:E}
Let $\sigma\in\cF$ be an irreducible unipotent representation of $G_0^F$ with
irreducible Alvis--Curtis dual $\sigma'$, and let $\pi$, $\tF$ and $(x^+,x^-)$
be as above. If the partition of $\cO^{\st}_{\cF}$ is not very even, then
\begin{equation}\label{eq:AC}
\sO(\sigma')^{\mx}=\sO\bigl(\Phi_{\tF}(x^+)\bigr),
\end{equation}
where the fiber is taken with respect to $\bS$ inside $\cO^{\st}_{\cF}$; the
right-hand side does not depend on the choice of $\pi$. If the partition is very
even, then $\sO(\sigma')^{\mx}=(\cO^{\st}_{\cF})^F$.
\end{theorem}

Applying Theorem~\ref{thm:E} to $\sigma'$ in place of $\sigma$ gives the
corresponding formula for $\sO(\sigma)^{\mx}$.

\subsection{Outline}

In Section~\ref{sec:orbits}, we recall the classification of rational
nilpotent orbits for special orthogonal groups and introduce the
rational partial order on these orbits.
In Section~\ref{sec:models}, we study how generalized Whittaker
models behave when representations of full orthogonal groups
are restricted to special orthogonal groups, including the
effect of conjugation by elements outside the latter.
In Section~\ref{sec:max}, we determine the maximal rational
wavefront sets and compute the corresponding Whittaker
multiplicities, with particular attention to the very even case.
{
In Section~\ref{sec:closure}, we prove that the rational
wavefront set of an irreducible representation coincides with
the rational closure of its maximal rational wavefront set.
Finally, in Section~\ref{sec:unipotent}, we use Lusztig families and
canonical quotients to give a more precise description of
the maximal rational wavefront sets of unipotent representations
and their behavior under Alvis--Curtis duality.}
% ===============================================================
\section{Rational nilpotent orbits and the rational order}\label{sec:orbits}
% ===============================================================

\subsection{Multiplicity spaces and the determinant}

We retain the notation of the introduction for $G$, $G_0$ and their
$\kk$-structures. Choose an $\sltwo(\kk)$-triple
$\{X,H,Y\}\subset\g^F=\so(V)$ with $X\in\Nil^F$. The corresponding
{$\sltwo(\kk)$-action} gives an orthogonal decomposition over $\kk$
\begin{equation}\label{eq:isotypic}
V=\bigoplus_{d\geqslant1}L_d\otimes_{\kk}W_d,
\end{equation}
where $L_d$ is the {irreducible $\sltwo(\kk)$-module} of dimension $d$ and $W_d$
is its multiplicity space over $\kk$. Thus $m_d=\dim_{\kk}W_d$ is the
multiplicity of $d$ in the partition of $X$. Fix a nondegenerate $\sltwo(\kk)$-invariant bilinear form
$\beta_d$ on $L_d$. It is symmetric for odd $d$ and {skew-symmetric} for even $d$. The
\rev{restriction of $B$ to} $L_d\otimes W_d$ is the tensor product of
$\beta_d$ with a nondegenerate form on $W_d$. This form is symmetric for odd
$d$ and {skew-symmetric} for even $d$; in particular, $m_d$ is even when $d$ is
even. Recall that over $\kk$ a
nondegenerate {skew-symmetric} space is determined up to isometry by its dimension,
and a nondegenerate symmetric space by its dimension and discriminant.

Following \cite[Section~3.1]{Wang}, let $\mathcal{Y}(V)$ be the set of signed
Young diagrams
\[
D=[(\lambda_1^{n_1},a_1),\dots,(\lambda_t^{n_t},a_t)]
\]
such that
\[
\lambda_1>\dots>\lambda_t>0,\quad n_i\in\mathbb{Z}_{>0},\quad
\sum_i n_i\lambda_i=\dim V,\quad a_i\in \kk^\times/\kk^{\times2}=\{\pm\},
\]
\[
\lambda_i\text{ even}\implies n_i\text{ even and }a_i=+,\qquad
\prod_{i<j}\varepsilon(-1)^{n_in_j}\prod_i a_i=\disc(V).
\]
\finalrev{For odd $\lambda_i$, the associated symmetric space has dimension
$n_i$, and its two possible types are distinguished by the
discriminant $a_i$ \cite[Section~3.1]{Wang}. For even $\lambda_i$,
the associated skew-symmetric space has dimension $n_i$ and only
one possible type, so we always write $a_i=+$.
For convenience, we sometimes omit the sign $a_i=+$ when
$\lambda_i$ is even.

We use this parametrization throughout. For $X\in\Nil^F$,
the integers $n_i$ are the multiplicities
$\dim W_{\lambda_i}$ in \eqref{eq:isotypic}, and the signs
$a_i$ are those attached to $X$ by this parametrization.
Thus there is a bijection
}\[
\Nil^F/\Ad(G^F)\xrightarrow{\ \sim\ }\mathcal{Y}(V),\qquad
{\cO=\Ad(G^F)X}{\longmapsto D=}[(\lambda_1^{n_1},a_1),\dots,(\lambda_t^{n_t},a_t)].
\]
{We denote $\cO$ by $\cO_D$ if it corresponds to the signed Young diagram $D$. If the underlying stable orbit of $\cO$ contains a unique rational orbit, then we denote $\cO$ by $\cO_\lambda$, where $\lambda$ is the partition of the underlying stable orbit.}

\rev{For $G_0^F$ this parametrization differs only at very even partitions
(see Lemma~\ref{lem:split} below). A very even partition $\lambda$ occurs only when
$V$ is split, and its signed Young diagram then corresponds to two rational
$G_0^F$-orbits, which we denote $\cO_\lambda^{\I}$ and $\cO_\lambda^{\II}$.
They lie in distinct stable orbits, each containing exactly one rational
$G_0^F$-orbit. Thus rational $G_0^F$-orbits are parametrized by signed Young
diagrams, together with a label $\I$ or $\II$ in the very even case.}

Write $W_{d,\kb}=W_d\otimes_{\kk}\kb$ for the base change of a
multiplicity space. The algebraic centralizer of {the $\sltwo(\kk)$-triple
$\{X,H,Y\}\subset\g^F=\so(V)$} in $G$ is
\begin{equation}\label{eq:reductive-centralizer}
\prod_{d\text{ odd}}\mathrm{O}(W_{d,\kb})\times\prod_{d\text{ even}}\mathrm{Sp}(W_{d,\kb}),
\end{equation}
with its induced Frobenius structure. Its $F$-fixed points are the product
of the isometry groups of the $\kk$-spaces $W_d$. The determinant of an
element $(g_d)$ acting on $V_{\kb}$ is
\begin{equation}\label{eq:det}
\det\nolimits_{V_{\kb}}\bigl((g_d)\bigr)=\prod_d\det(g_d)^d=\prod_{d\text{ odd}}\det(g_d).
\end{equation}
Since the unipotent radical of $C_G(X)$ is connected and consists of elements of
determinant one, \eqref{eq:det} also determines whether $C_G(X)$ contains an
element of determinant $-1$.

\begin{lemma}\label{lem:split}
A rational $G^F$-orbit is the union of two rational $G_0^F$-orbits if its
partition is very even, and is a single rational $G_0^F$-orbit otherwise. In
the very even case $V$ is split, and each of the two stable $G_0$-orbits
\rev{with this partition contains exactly one rational $G_0^F$-orbit.}
\end{lemma}

\begin{proof}
The number of $G_0^F$-orbits in $\Ad(G^F)X$ equals
$[G^F:G_0^FC_{G^F}(X)]$.

{Suppose that an odd part occurs in the partition of $X$.} Fix such a part $d$ and a vector
$v\in W_d$ with $(v,v)\neq0$, where $(\ ,\ )$ is the symmetric form on
$W_d$. Define $z\colon V\to V$ by
\[
z(u\otimes w)=
\begin{cases}
u\otimes\Bigl(w-\dfrac{2(w,v)}{(v,v)}\,v\Bigr),&u\in L_d,\ w\in W_d,\\[1.5ex]
u\otimes w,&u\in L_e,\ w\in W_e,\ e\neq d,
\end{cases}
\]
and extend $z$ linearly using \eqref{eq:isotypic}. Since
$W_d=\kk v\oplus v^\perp$, the map $z$ acts by $-1$ on $L_d\otimes\kk v$ and by $+1$ on its orthogonal
complement in $V$. Hence $z$ is an isometry defined over $\kk$, so $z\in G^F$. On each summand $L_e\otimes W_e$ the operator $X$ acts through
$L_e$, so $zX=Xz$ and $z\in C_{G^F}(X)$. The $(-1)$-eigenspace of $z$ has dimension $d$, so $\det z=(-1)^d=-1$. {Hence $z\notin G_0^F$.}
Therefore $G^F=G_0^FC_{G^F}(X)$, and the index is $1$.

{Suppose now that all parts of the partition of $X$ are even.} Then every factor in
\eqref{eq:reductive-centralizer} is symplectic, so $C_G(X)$ is connected and
contained in $G_0$. Hence $G_0^FC_{G^F}(X)=G_0^F$, and the index is
$[G^F:G_0^F]=2$. Moreover, $V$ is split: if $A_d\subset L_d$ is a Lagrangian
subspace, then $A_d\otimes W_d$ is a totally isotropic subspace of half the
dimension of $L_d\otimes W_d$. Their sum is a totally isotropic subspace of
$V$ of half its dimension.

Over $\kb$ the same argument shows that $\Ad(G(\kb))X$ is the union of two
$G_0$-orbits {if all parts of the partition of $X$ are even.} Both contain rational points, namely $X$ and $cXc^{-1}$ for any
rational isometry $c$ of determinant $-1$. Their centralizers in $G_0$ are connected. By Lang's theorem (equivalently,
by the parametrization of rational forms by $H^1(F,C_{G_0}(X))$), each
contains a single $G_0^F$-orbit.
\end{proof}

\begin{lemma}\label{lem:conj}
Fix $c\in G^F\setminus G_0^F$. Conjugation by $c$ preserves every rational
$G_0^F$-orbit whose partition is not very even, and interchanges
$\cO_\lambda^{\I}$ and $\cO_\lambda^{\II}$ for every very even partition
$\lambda$. The same holds for the corresponding stable
$G_0$-orbits.
\end{lemma}

\begin{proof}
Let $X\in\Nil^F$ have a partition that is not very even, and let $z$ be the
element constructed in the proof of Lemma~\ref{lem:split}. Then
$z\in C_{G^F}(X)$ and $\det z=-1$, so $h=cz\in G_0^F$ and
$hXh^{-1}=cXc^{-1}$. Thus conjugation by $c$ preserves $\Ad(G_0^F)X$.

Now let $X$ have very even partition $\lambda$. If $cXc^{-1}=hXh^{-1}$ for some
$h\in G_0^F$, then $h^{-1}c\in C_{G^F}(X)$ would have determinant $-1$,
contradicting $C_{G^F}(X)\subset G_0^F$ (proof of Lemma~\ref{lem:split}). Since
$\Ad(G^F)X=\cO_\lambda^{\I}\cup\cO_\lambda^{\II}$, conjugation by $c$
interchanges the two orbits. The same arguments over $\kb$ prove the assertions for stable orbits.
\end{proof}

\begin{example}\label{ex:SO4}
Let $V$ be the split symmetric space of dimension $4$ with hyperbolic basis
$(e_1,e_2,f_1,f_2)$, and let $c$ interchange $e_1$ and $f_1$ and fix $e_2$ and
$f_2$; then $c\in G^F\setminus G_0^F$. Put
\[
X=\begin{pmatrix}0&0&0&1\\0&0&-1&0\\0&0&0&0\\0&0&0&0\end{pmatrix},\qquad
cXc^{-1}=\begin{pmatrix}0&0&0&0\\-1&0&0&0\\0&0&0&1\\0&0&0&0\end{pmatrix}.
\]
Both matrices are skew-adjoint with respect to the hyperbolic form, have square
zero and rank two, and hence have partition $(2^2)$. They are conjugate under
$\mathrm{O}^+_4(\kk)$ via $c$ but, by Lemma~\ref{lem:split}, not under
$\mathrm{SO}^+_4(\kk)$. They represent the two rational $G_0^F$-orbits with
partition $(2^2)$.
\end{example}

\subsection{Partial order on rational orbits}

We first recall the order on rational orbits for the full orthogonal group.
For a signed Young diagram
\[
D=[(\lambda_1^{n_1},a_1),\dots,(\lambda_t^{n_t},a_t)],\qquad
\lambda_1>\dots>\lambda_t,
\]
put
\[
c_r(D)=\sum_{\lambda_i\geqslant r}n_i,\qquad
C_r(D)=\sum_{j=1}^r c_j(D),\qquad
\sgn_r(D)=\prod_{\lambda_i\geqslant r}a_i
\prod_{\substack{i<j\\\lambda_i,\lambda_j\geqslant r}}\varepsilon(-1)^{n_in_j}.
\]
The order defined in \cite[Section~3.2]{Wang} is
\begin{equation}\label{eq:full-order}
D\leqslant_G E\iff
\begin{cases}
C_r(D)\geqslant C_r(E)&\text{for every }r\geqslant1,\\
\sgn_{2s}(D)=\sgn_{2s}(E)&\text{whenever }C_{2s}(D)=C_{2s}(E).
\end{cases}
\end{equation}
The first condition is the dominance order on partitions, written in terms of
columns; the second is the rational sign condition.

\begin{definition}\label{def:order}
For rational $G_0^F$-orbits $\cO_0,\cO_1$ we write $\cO_0\leqslant_{G_0}\cO_1$
if
\begin{equation}\label{eq:SO-order}
\pr(\cO_0)\leqslant_G\pr(\cO_1)\qquad\text{and}\qquad
\cO_0^{\st}\subset\overline{\cO_1^{\st}},
\end{equation}
where the bar denotes Zariski closure over $\kb$. The \emph{rational closure}
of a union $\mathcal U\subset\Nil^F$ of rational orbits is
\[
\overline{\mathcal U}=\bigcup_{\substack{\cO_0\leqslant_{G_0}\cO_1\\\cO_1\in\mathcal U}}\cO_0 .
\]

\end{definition}

\rev{The relation \eqref{eq:SO-order} is a partial order on
$\Nil^F/\Ad(G_0^F)$. Reflexivity and transitivity follow from those of
$\leqslant_G$ and of the closure order on {stable} orbits. If
$\cO_0\leqslant_{G_0}\cO_1$ and $\cO_1\leqslant_{G_0}\cO_0$, then
$\pr(\cO_0)=\pr(\cO_1)$ and $\cO_0^{\st}=\cO_1^{\st}$, so $\cO_0=\cO_1$ by
Lemma~\ref{lem:split}. The closure condition makes the two very even orbits
with the same partition incomparable, since their stable orbits are distinct
and have the same dimension. Apart from this, $\leqslant_{G_0}$ is determined
by $\leqslant_G$ via $\pr$:}

\begin{lemma}\label{lem:order}
Let $\cO_0,\cO_1$ be rational $G_0^F$-orbits with partitions
$\lambda_0,\lambda_1$.
\begin{enumerate}[label=\textup{(\roman*)}]
\item $\cO_0<_{G_0}\cO_1$ if and only if $\pr(\cO_0)<_G\pr(\cO_1)$.
\item If $\lambda_0$ or $\lambda_1$ is not very even, then
$\cO_0\leqslant_{G_0}\cO_1$ if and only if $\pr(\cO_0)\leqslant_G\pr(\cO_1)$.
\end{enumerate}
\end{lemma}

\begin{proof}
(i) Suppose that $\pr(\cO_0)<_G\pr(\cO_1)$. \rev{Since a signed Young
diagram for $V$ is determined by its partition and the signs $\sgn_{2s}$,
\eqref{eq:full-order} forces $\lambda_0<\lambda_1$ in the dominance order.} By
\cite[Section~2.1]{GPST} this gives $\cO_0^{\st}\subset\overline{\cO_1^{\st}}$,
and hence $\cO_0<_{G_0}\cO_1$. Conversely, $\cO_0<_{G_0}\cO_1$ implies
$\pr(\cO_0)\leqslant_G\pr(\cO_1)$. If equality held, the two $G_0^F$-orbits
would either coincide or be the two very even orbits with the same partition.
In neither case can $\cO_0<_{G_0}\cO_1$ hold. Hence
$\pr(\cO_0)<_G\pr(\cO_1)$.

(ii) The forward implication is part of Definition~\ref{def:order}.
Conversely, assume $\pr(\cO_0)\leqslant_G\pr(\cO_1)$. If the inequality is
strict, apply (i). Otherwise $\lambda_0=\lambda_1$ is not very even, and
$\cO_0=\cO_1$ by Lemma~\ref{lem:split}.
\end{proof}

\begin{example}[{$\dim V=8$}] \label{ex:SO8}

\emph{Split form.} For $\mathrm{O}_8^+(\kk)$ there are $16$ rational $G^F$-orbits.
Passing to $\mathrm{SO}_8^+(\kk)$, each of the orbits with partition $(4^2)$ or
$(2^4)$ splits into two, labeled $\I$ and $\II$, giving $18$ rational
$G_0^F$-orbits. Figure~\ref{fig:O8plus} compares the two Hasse diagrams. The
orbits $(4^2)^{\I}$ and $(4^2)^{\II}$ are incomparable, since their stable
orbits are distinct and have the same dimension. Both lie below the two orbits
with partition $(5,3)$ and above the two orbits with partition $(3^2,1^2)$. The same
holds for $(2^4)^{\I}$ and $(2^4)^{\II}$ relative to $(3,2^2,1)$ and
$(2^2,1^4)$. {Every other relation for $\mathrm{O}_8^+(\kk)$ lifts unchanged.}

\emph{Non-split form.} For $\mathrm{O}_8^-(\kk)$ no very even partition occurs, so
$\pr$ is a bijection identifying $\leqslant_{G_0}$ with $\leqslant_G$ by
Lemma~\ref{lem:order}(ii). Hence $\mathrm{O}_8^-(\kk)$ and $\mathrm{SO}_8^-(\kk)$ have the
same Hasse diagram, with $14$ rational orbits; see Figure~\ref{fig:O8minus}.

\end{example}

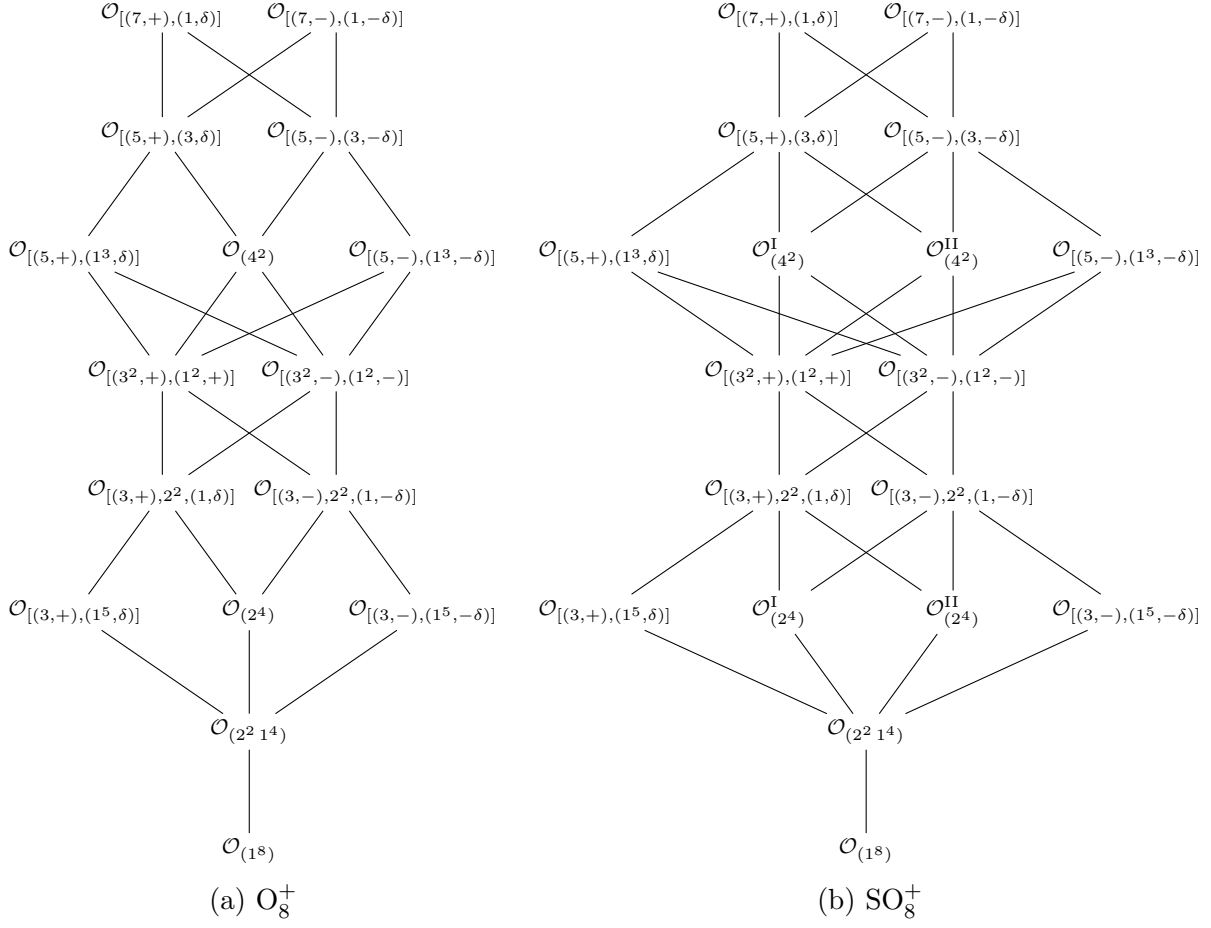
\begin{figure}[H]
\centering
\begin{minipage}[b]{0.42\textwidth}\centering
\begin{tikzpicture}[orb/.style={inner sep=2pt,font=\scriptsize,fill=white}]
  \node[orb] (a0) at (1.15,11.06) {${\cO_{[(7,+),(1,\delta)]}}$};
  \node[orb] (a1) at (3.45,11.06) {${\cO_{[(7,-),(1,-\delta)]}}$};
  \node[orb] (a2) at (1.15,9.48) {${\cO_{[(5,+),(3,\delta)]}}$};
  \node[orb] (a3) at (3.45,9.48) {${\cO_{[(5,-),(3,-\delta)]}}$};
  \node[orb] (a4) at (0.00,7.90) {${\cO_{[(5,+),(1^{3},\delta)]}}$};
  \node[orb] (a5) at (4.60,7.90) {${\cO_{[(5,-),(1^{3},-\delta)]}}$};
  \node[orb] (a6) at (2.30,7.90) {${\cO_{(4^{2})}}$};
  \node[orb] (a7) at (1.15,6.32) {${\cO_{[(3^{2},+),(1^{2},+)]}}$};
  \node[orb] (a8) at (3.45,6.32) {${\cO_{[(3^{2},-),(1^{2},-)]}}$};
  \node[orb] (a9) at (1.15,4.74) {${\cO_{[(3,+),2^{2},(1,\delta)]}}$};
  \node[orb] (a10) at (3.45,4.74) {${\cO_{[(3,-),2^{2},(1,-\delta)]}}$};
  \node[orb] (a11) at (0.00,3.16) {${\cO_{[(3,+),(1^{5},\delta)]}}$};
  \node[orb] (a12) at (4.60,3.16) {${\cO_{[(3,-),(1^{5},-\delta)]}}$};
  \node[orb] (a13) at (2.30,3.16) {${\cO_{(2^{4})}}$};
  \node[orb] (a14) at (2.30,1.58) {${\cO_{(2^{2}\,1^{4})}}$};
  \node[orb] (a15) at (2.30,0.00) {${\cO_{(1^{8})}}$};
  \draw (a2) -- (a1);
  \draw (a2) -- (a0);
  \draw (a3) -- (a1);
  \draw (a3) -- (a0);
  \draw (a4) -- (a2);
  \draw (a5) -- (a3);
  \draw (a6) -- (a2);
  \draw (a6) -- (a3);
  \draw (a7) -- (a6);
  \draw (a7) -- (a5);
  \draw (a7) -- (a4);
  \draw (a8) -- (a6);
  \draw (a8) -- (a5);
  \draw (a8) -- (a4);
  \draw (a9) -- (a7);
  \draw (a9) -- (a8);
  \draw (a10) -- (a7);
  \draw (a10) -- (a8);
  \draw (a11) -- (a9);
  \draw (a12) -- (a10);
  \draw (a13) -- (a9);
  \draw (a13) -- (a10);
  \draw (a14) -- (a11);
  \draw (a14) -- (a13);
  \draw (a14) -- (a12);
  \draw (a15) -- (a14);
\end{tikzpicture}

\smallskip
\textup{(a)} $\mathrm{O}_8^+$
\end{minipage}\hfill
\begin{minipage}[b]{0.56\textwidth}\centering
\begin{tikzpicture}[orb/.style={inner sep=2pt,font=\scriptsize,fill=white}]
  \node[orb] (b0) at (2.30,11.06) {${\cO_{[(7,+),(1,\delta)]}}$};
  \node[orb] (b1) at (4.60,11.06) {${\cO_{[(7,-),(1,-\delta)]}}$};
  \node[orb] (b2) at (2.30,9.48) {${\cO_{[(5,+),(3,\delta)]}}$};
  \node[orb] (b3) at (4.60,9.48) {${\cO_{[(5,-),(3,-\delta)]}}$};
  \node[orb] (b4) at (0.00,7.90) {${\cO_{[(5,+),(1^{3},\delta)]}}$};
  \node[orb] (b5) at (6.90,7.90) {${\cO_{[(5,-),(1^{3},-\delta)]}}$};
  \node[orb] (b6) at (2.30,7.90) {${\cO_{(4^{2})}^{\mathrm{I}}}$};
  \node[orb] (b7) at (4.60,7.90) {${\cO_{(4^{2})}^{\mathrm{II}}}$};
  \node[orb] (b8) at (2.30,6.32) {${\cO_{[(3^{2},+),(1^{2},+)]}}$};
  \node[orb] (b9) at (4.60,6.32) {${\cO_{[(3^{2},-),(1^{2},-)]}}$};
  \node[orb] (b10) at (2.30,4.74) {${\cO_{[(3,+),2^{2},(1,\delta)]}}$};
  \node[orb] (b11) at (4.60,4.74) {${\cO_{[(3,-),2^{2},(1,-\delta)]}}$};
  \node[orb] (b12) at (0.00,3.16) {${\cO_{[(3,+),(1^{5},\delta)]}}$};
  \node[orb] (b13) at (6.90,3.16) {${\cO_{[(3,-),(1^{5},-\delta)]}}$};
  \node[orb] (b14) at (2.30,3.16) {${\cO_{(2^{4})}^{\mathrm{I}}}$};
  \node[orb] (b15) at (4.60,3.16) {${\cO_{(2^{4})}^{\mathrm{II}}}$};
  \node[orb] (b16) at (3.45,1.58) {${\cO_{(2^{2}\,1^{4})}}$};
  \node[orb] (b17) at (3.45,0.00) {${\cO_{(1^{8})}}$};
  \draw (b2) -- (b1);
  \draw (b2) -- (b0);
  \draw (b3) -- (b1);
  \draw (b3) -- (b0);
  \draw (b4) -- (b2);
  \draw (b5) -- (b3);
  \draw (b6) -- (b2);
  \draw (b6) -- (b3);
  \draw (b7) -- (b2);
  \draw (b7) -- (b3);
  \draw (b8) -- (b5);
  \draw (b8) -- (b6);
  \draw (b8) -- (b4);
  \draw (b8) -- (b7);
  \draw (b9) -- (b5);
  \draw (b9) -- (b6);
  \draw (b9) -- (b4);
  \draw (b9) -- (b7);
  \draw (b10) -- (b8);
  \draw (b10) -- (b9);
  \draw (b11) -- (b8);
  \draw (b11) -- (b9);
  \draw (b12) -- (b10);
  \draw (b13) -- (b11);
  \draw (b14) -- (b10);
  \draw (b14) -- (b11);
  \draw (b15) -- (b10);
  \draw (b15) -- (b11);
  \draw (b16) -- (b14);
  \draw (b16) -- (b12);
  \draw (b16) -- (b15);
  \draw (b16) -- (b13);
  \draw (b17) -- (b16);
\end{tikzpicture}

\smallskip
\textup{(b)} $\mathrm{SO}_8^+$
\end{minipage}
\caption{The rational orders $\leqslant_G$ for $\mathrm{O}_8^+$ and
$\leqslant_{G_0}$ for $\mathrm{SO}_8^+$, with $q$ odd and
$\delta=\varepsilon(-1)$. Larger orbits are higher.}
\label{fig:O8plus}
\end{figure}

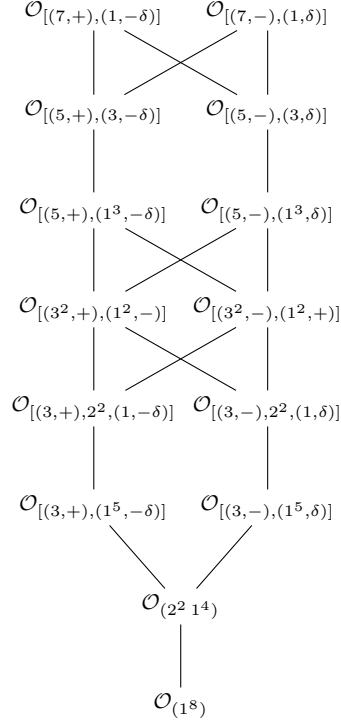
\begin{figure}[H]
\centering
\begin{tikzpicture}[orb/.style={inner sep=2pt,font=\scriptsize,fill=white}]
  \node[orb] (c0) at (0.00,9.16) {${\cO_{[(7,+),(1,-\delta)]}}$};
  \node[orb] (c1) at (2.30,9.16) {${\cO_{[(7,-),(1,\delta)]}}$};
  \node[orb] (c2) at (0.00,7.85) {${\cO_{[(5,+),(3,-\delta)]}}$};
  \node[orb] (c3) at (2.30,7.85) {${\cO_{[(5,-),(3,\delta)]}}$};
  \node[orb] (c4) at (0.00,6.55) {${\cO_{[(5,+),(1^{3},-\delta)]}}$};
  \node[orb] (c5) at (2.30,6.55) {${\cO_{[(5,-),(1^{3},\delta)]}}$};
  \node[orb] (c6) at (0.00,5.24) {${\cO_{[(3^{2},+),(1^{2},-)]}}$};
  \node[orb] (c7) at (2.30,5.24) {${\cO_{[(3^{2},-),(1^{2},+)]}}$};
  \node[orb] (c8) at (0.00,3.93) {${\cO_{[(3,+),2^{2},(1,-\delta)]}}$};
  \node[orb] (c9) at (2.30,3.93) {${\cO_{[(3,-),2^{2},(1,\delta)]}}$};
  \node[orb] (c10) at (0.00,2.62) {${\cO_{[(3,+),(1^{5},-\delta)]}}$};
  \node[orb] (c11) at (2.30,2.62) {${\cO_{[(3,-),(1^{5},\delta)]}}$};
  \node[orb] (c12) at (1.15,1.31) {${\cO_{(2^{2}\,1^{4})}}$};
  \node[orb] (c13) at (1.15,0.00) {${\cO_{(1^{8})}}$};
  \draw (c2) -- (c0);
  \draw (c2) -- (c1);
  \draw (c3) -- (c0);
  \draw (c3) -- (c1);
  \draw (c4) -- (c2);
  \draw (c5) -- (c3);
  \draw (c6) -- (c4);
  \draw (c6) -- (c5);
  \draw (c7) -- (c4);
  \draw (c7) -- (c5);
  \draw (c8) -- (c6);
  \draw (c8) -- (c7);
  \draw (c9) -- (c6);
  \draw (c9) -- (c7);
  \draw (c10) -- (c8);
  \draw (c11) -- (c9);
  \draw (c12) -- (c10);
  \draw (c12) -- (c11);
  \draw (c13) -- (c12);
\end{tikzpicture}
\caption{The rational order for $\mathrm{O}_8^-$, which coincides with that
for $\mathrm{SO}_8^-$, with $q$ odd and $\delta=\varepsilon(-1)$.
Larger orbits are higher.}
\label{fig:O8minus}
\end{figure}

% ===============================================================
\section{Whittaker models, restriction and conjugation}\label{sec:models}
% ===============================================================

\subsection{Generalized Whittaker models}

Fix an $\sltwo(\kk)$-triple $\{X,H,Y\}\subset\g^F$ as in
Section~\ref{sec:orbits}, and let $\gamma\colon\mathbb G_m\to G_0$ be the
$\kk$-defined cocharacter obtained from the diagonal torus of the
corresponding $\mathrm{SL}_2$-action. Thus $d\gamma(1)=H$. \rev{Write}
\[
\g_j=\{Z\in\g:\Ad(\gamma(t))Z=t^jZ\text{ for all }t\in\kb^\times\},
\qquad \g_{\geqslant r}=\bigoplus_{j\geqslant r}\g_j.
\]
\rev{The weights $j$ are integers, the subspaces $\g_j$ are $F$-stable, and
$\ad H$ acts on $\g_j$ by multiplication by the image of $j$ in $\kk$.}
Let $G_{\geqslant1}$ be the connected unipotent subgroup with Lie algebra
$\g_{\geqslant1}$. \finalrev{We use the $G$-invariant nondegenerate symmetric bilinear form
\[
\langle Z,Z'\rangle=\operatorname{tr}_{V_{\kb}}(ZZ')
\qquad (Z,Z'\in\g).
\]
This form is defined over $\kk$ and is nondegenerate since $p\neq2$.}
The pairing $(Z,Z')\mapsto\langle Y,[Z,Z']\rangle$ on $\g_1$ is alternating and
nondegenerate. Choose an $F$-stable Lagrangian \rev{subspace} $\mathfrak{l}\subset\g_1$, and let
$G_{\geqslant1.5}$ be the connected unipotent subgroup with Lie algebra
$\g_{\geqslant1.5}=\mathfrak{l}\oplus\g_{\geqslant2}$.

These unipotent subgroups are defined over $\kk$. The form $\langle\ ,\ \rangle$ is $\kb$-bilinear and takes values in
$\kk$ on $\g^F\times\g^F$.

For $J=G,G_0$, let $\cO^J_X=\Ad(J^F)X$.
Fix the Cayley map $\Cay\colon\mathcal U(G)\to\Nil$ defined by
\[
\Cay(u)=(u-I)(u+I)^{-1},
\qquad u\in\mathcal U(G),
\]
where $I$ is the identity operator on $V_{\kb}$.
This map is $G$-equivariant and $F$-compatible.
This map sends $G_{\geqslant1.5}^F$ into
$\g_{\geqslant1.5}^F$. Define
\[
\psi_{\cO^J_X}(u):=\psi\bigl(\langle Y,\Cay(u)\rangle\bigr),\qquad
u\in G_{\geqslant1.5}^F.
\]
Since $G$ and $G_0$ have the same Lie algebra and every unipotent isometry has
determinant one, $G_{\geqslant1.5}\subset G_0$; the group
$G_{\geqslant1.5}^F$ and the character $\psi_{\cO^J_X}$ are therefore independent of whether
$J=G$ or $J=G_0$. In particular $\psi_{\cO^G_X}=\psi_{\cO^{G_0}_X}$.
Following \cite[Section~3]{Wang}, define the generalized Gelfand--Graev representation by
\begin{equation}\label{eq:GGGR}
\Gamma^J_{\cO^J_X}:=\Ind^{J^F}_{G_{\geqslant1.5}^F}\psi_{\cO^J_X},
\end{equation}
and, for a representation $\tau$ of $J^F$, the generalized Whittaker model by
\begin{equation}\label{eq:Wh}
\Wh_{\cO^J_X}(\tau):=\Hom_{G_{\geqslant1.5}^F}(\psi_{\cO^J_X},\tau).
\end{equation}
The isomorphism class of $\Gamma^J_{\cO^J_X}$ depends only on the rational
orbit, not on the choice of triple or Lagrangian. By Frobenius reciprocity,
$\dim\Wh_{\cO^J_X}(\tau)=\langle\Gamma^J_{\cO^J_X},\tau\rangle_{J^F}$.

For $\cO=\cO^J_X$, the model $\Wh_{\cO}(\tau)$ carries an action of
$M^F_{\psi_{\cO}}$, where, in the notation of \cite[Section~3]{Wang},
\[
M_{\cO}=C_J(H),\qquad
M_{\psi_{\cO}}=C_{M_{\cO}}(Y)=C_J(\{X,H,Y\}).
\]
These groups are defined over $\kk$, and the action is that of the finite group
$M_{\psi_{\cO}}^F$. As in \cite[Section~3]{Wang}, let $\omega_{\psi_{\cO}}$
be the canonical representation of $M^F_{\psi_{\cO}}\ltimes G_{\geqslant1}^F$
extending $\Ind^{G_{\geqslant1}^F}_{G_{\geqslant1.5}^F}\psi_{\cO}$: it is the
associated Weil representation if $\g_1\neq0$, and the extension of
$\psi_{\cO}$ trivial on $M^F_{\psi_{\cO}}$ if $\g_1=0$. Under the
identification
$\Wh_{\cO}(\tau)\simeq\Hom_{G_{\geqslant1}^F}(\omega_{\psi_{\cO}},\tau)$, the
action is
\[
(m\cdot f)(v)=\tau(m)\,f\bigl(\omega_{\psi_{\cO}}(m)^{-1}v\bigr),\qquad
m\in M^F_{\psi_{\cO}},
\]
for $f\in\Hom_{G_{\geqslant1}^F}(\omega_{\psi_{\cO}},\tau)$ and $v$ in the
space of $\omega_{\psi_{\cO}}$.

\begin{proposition}\label{prop:restriction}
For every representation $\tau$ of $G^F$ and every $X\in\Nil^F$,
\begin{equation}\label{eq:Wh-res}
\Wh_{\cO^G_X}(\tau)=\Wh_{\cO^{G_0}_X}\bigl(\ResG\tau\bigr)
\end{equation}
as vector spaces. Moreover,
\[
\Gamma^G_{\cO^G_X}\simeq\IndG\Gamma^{G_0}_{\cO^{G_0}_X},\qquad
\Wh_{\cO^G_X}(\tau\otimes\det)\simeq\Wh_{\cO^G_X}(\tau).
\]
In particular, twisting by $\det$ preserves all Whittaker multiplicities.
If $\tau$ is irreducible, it also preserves $\sO(\tau)$.
\end{proposition}

\begin{proof}
Both sides of \eqref{eq:Wh-res} are the same space of
$G_{\geqslant1.5}^F$-homomorphisms. The formula for $\Gamma^G_{\cO^G_X}$ follows from
induction in stages, and the last isomorphism follows because $\det$ is
trivial on $G_{\geqslant1.5}^F$.
\end{proof}

\begin{lemma}\label{lem:pr-WF}
Let $\pi\in\Irr(G^F)$, and let $\sigma\in\Irr(G_0^F)$ occur in $\ResG\pi$. For
every rational $G_0^F$-orbit $\cO_0$,
\[
\cO_0\in\sO(\sigma)\implies\pr(\cO_0)\in\sO(\pi).
\]
\end{lemma}

\begin{proof}
Take $X\in\cO_0$. {Both generalized Gelfand--Graev representations $\Gamma^G_{\cO^G_X}$ and $\Gamma^{G_0}_{\cO^{G_0}_X}$ use} the subgroup $G_{\geqslant1.5}^F$ and
the character $\psi_{\cO_0}=\psi_{\pr(\cO_0)}$. Thus an embedding
$\sigma\hookrightarrow\ResG\pi$ induces an injection
$\Wh_{\cO_0}(\sigma)\hookrightarrow\Wh_{\pr(\cO_0)}(\pi)$.
\end{proof}

\subsection{Clifford theory for index-two subgroups}

Let $c$ be a rational reflection, so that $c^2=1$ and
$G^F=G_0^F\rtimes\langle c\rangle$. For a representation $\sigma$ of $G_0^F$,
put $\sigma^c(h)=\sigma(c^{-1}hc)$. The following is a consequence of Clifford theory; see
\cite[Chapter~6]{Isaacs}.

\begin{lemma}\label{lem:clifford}
Let $\pi\in\Irr(G^F)$.
\begin{enumerate}[label=\textup{(\roman*)}]
\item If $\pi\otimes\det\not\simeq\pi$, then $\ResG\pi=\sigma$ is irreducible,
and the two extensions of $\sigma$ to $G^F$ are $\pi$ and $\pi\otimes\det$.
\item If $\pi\otimes\det\simeq\pi$, then $\ResG\pi=\sigma^{\I}\oplus\sigma^{\II}$
where $\sigma^{\I},\sigma^{\II}\in\Irr(G_0^F)$ are inequivalent. Moreover,
$\sigma^{\II}\simeq(\sigma^{\I})^c$ and $\IndG\sigma^{\I}\simeq\pi$.
\end{enumerate}
Every irreducible representation of $G_0^F$ arises in one of these two ways.
\end{lemma}

\begin{proposition}\label{prop:conj}
For $\sigma\in\Irr(G_0^F)$,
\begin{equation}\label{eq:Wh-conj}
\Wh_{\Ad(c)\cO^{G_0}_X}(\sigma^c)\simeq\Wh_{\cO^{G_0}_X}(\sigma).
\end{equation}
Conjugation by $c$ identifies $M^F_{\psi_{\cO^{G_0}_X}}$ with
$M^F_{\psi_{\cO^{G_0}_{cXc^{-1}}}}$, and \eqref{eq:Wh-conj} is equivariant for
their actions. Consequently, $\sO(\sigma^c)=\Ad(c)\sO(\sigma)$ and
$\cO_{\sigma^c}^{\st}=\Ad(c)\cO_{\sigma}^{\st}$.
\end{proposition}

\begin{proof}
For $cXc^{-1}$ we may use the subgroup $cG_{\geqslant1.5}c^{-1}$ in
\eqref{eq:Wh}. By the $G$-invariance of $\langle\ ,\ \rangle$ and the
equivariance of the Cayley map,
\[
\psi_{\cO^{G_0}_{cXc^{-1}}}(cuc^{-1})=\psi_{\cO^{G_0}_X}(u),\qquad
u\in G_{\geqslant1.5}^F.
\]
For $f\in\Hom_{G_{\geqslant1.5}^F}(\psi_{\cO^{G_0}_X},\sigma)$ and
$z\in\mathbb{C}$ we then have
\[
\sigma^c(cuc^{-1})f(z)=\sigma(u)f(z)
=f\bigl(\psi_{\cO^{G_0}_{cXc^{-1}}}(cuc^{-1})z\bigr),
\]
so $f$ belongs to the space on the left-hand side of \eqref{eq:Wh-conj}.
This gives \eqref{eq:Wh-conj}, with inverse induced by conjugation by $c^{-1}$.

Conjugation by $c$ identifies the centralizers of $\{X,H,Y\}$ and of
$\{cXc^{-1},cHc^{-1},cYc^{-1}\}$. It also identifies the symplectic spaces
underlying $\omega_{\psi_{\cO^{G_0}_X}}$ and
$\omega_{\psi_{\cO^{G_0}_{cXc^{-1}}}}$, so the Weil representations may be
chosen with
$\omega_{\psi_{\cO^{G_0}_{cXc^{-1}}}}(cmc^{-1})=\omega_{\psi_{\cO^{G_0}_X}}(m)$
for $m\in M^F_{\psi_{\cO^{G_0}_X}}$. Together with $\sigma^c(cmc^{-1})=\sigma(m)$,
the action formula above shows that \eqref{eq:Wh-conj} is equivariant.
As $X$ varies over the rational nilpotent elements, we obtain
$\sO(\sigma^c)=\Ad(c)\sO(\sigma)$. Finally, $\Ad(c)$ is an automorphism of
$\Nil$ and therefore preserves Zariski closure. The assertion for the
Kawanaka wavefront orbits follows.
\end{proof}

\begin{corollary}\label{cor:half}
Let $\pi\in\Irr(G^F)$ with $\ResG\pi=\sigma^{\I}\oplus\sigma^{\II}$. Let
$\cO_0$ be a rational $G_0^F$-orbit whose partition is not very even, and put
$\cO=\pr(\cO_0)$. Then
\begin{equation}\label{eq:half}
\dim\Wh_{\cO_0}(\sigma^{\I})=\dim\Wh_{\cO_0}(\sigma^{\II})
=\tfrac12\dim\Wh_{\cO}(\pi).
\end{equation}
In particular,
$\cO_0\in\sO(\sigma^{\I})\iff\cO_0\in\sO(\sigma^{\II})\iff\cO\in\sO(\pi)$.
\end{corollary}

\begin{proof}
By Lemma~\ref{lem:conj}, $\Ad(c)\cO_0=\cO_0$. Since
$\sigma^{\II}\simeq(\sigma^{\I})^c$, {by Proposition~\ref{prop:conj}, we have}
$\Wh_{\cO_0}(\sigma^{\I})\simeq\Wh_{\cO_0}(\sigma^{\II})$. On the other hand,
Proposition~\ref{prop:restriction} gives
\[
\dim\Wh_{\cO}(\pi)=\dim\Wh_{\cO_0}(\sigma^{\I})+\dim\Wh_{\cO_0}(\sigma^{\II}).
\qedhere
\]
\end{proof}

\subsection{Irreducible restrictions}\label{subsec:irred}

Suppose first that \rev{$\dim V$ is odd}. Then
$G^F=G_0^F\times\langle-I\rangle$, so
\[
\Irr(G^F)=\{\sigma\boxtimes\eta:\ \sigma\in\Irr(G_0^F),\ \eta\in\Irr(\langle-I\rangle)\},
\]
and for $\pi=\sigma\boxtimes\eta$ we have $\ResG\pi=\sigma$ and
$\pi(-I)=\eta(-I)\,\mathrm{id}$ with $\eta(-I)\in\{\pm1\}$. Moreover,
$\Ad(-I)X=X$ for all $X\in\g$, so $\Ad(G^F)X=\Ad(G_0^F)X$. \rev{In particular, every irreducible representation of $G^F$ has irreducible
restriction to $G_0^F$.}

For the rest of this subsection, let $\dim V$ be arbitrary and suppose
that $\pi\in\Irr(G^F)$ has irreducible restriction $\sigma=\ResG\pi$.
By Proposition~\ref{prop:restriction}, for every rational $G_0^F$-orbit
$\cO_0$ we have
\begin{equation}\label{eq:irred-res}
\dim\Wh_{\cO_0}(\sigma)=\dim\Wh_{\pr(\cO_0)}(\pi),\qquad
\sO(\sigma)=\pr^{-1}\bigl(\sO(\pi)\bigr).
\end{equation}

\begin{proposition}\label{prop:irred-max}
Let $\pi\in\Irr(G^F)$ have irreducible restriction $\sigma=\ResG\pi$, and let
$\lambda$ be the partition of $\cO_{\pi}^{\st}$. Then $\lambda$ is not very
even. Each $\cO\in\sO(\pi)^{\mx}$ has a unique preimage $\cO_0$ under $\pr$;
this preimage lies in $\sO(\sigma)$ and satisfies
$\cO_0^{\st}=\cO_{\sigma}^{\st}$. The Kawanaka wavefront orbit
$\cO_{\sigma}^{\st}$ is stable under $\Ad(c)$.
\end{proposition}

\begin{proof}
Suppose that $\lambda$ is very even. \rev{Then $\cO_\lambda$ is the only
rational $G^F$-orbit with partition $\lambda$, so $\cO_\lambda\in\sO(\pi)$, and
by \eqref{eq:irred-res} both} $\cO_\lambda^{\I}$ and $\cO_\lambda^{\II}$ belong
to $\sO(\sigma)$. By
Lemma~\ref{lem:pr-WF} and \cite[Theorem~1.1]{Wang}, every partition of an orbit
in $\sO(\sigma)$ is at most $\lambda$. The two distinct stable orbits
$(\cO_\lambda^{\I})^{\st}$ and $(\cO_\lambda^{\II})^{\st}$ would then both be
maximal among the stable orbits meeting $\sO(\sigma)$, contradicting the
uniqueness of the Kawanaka wavefront orbit \cite[Theorem~14.10]{Tay}.

{We now consider the case where $\lambda$ is not very even.}
Every $\cO\in\sO(\pi)^{\mx}$ has partition $\lambda$ by
\cite[Theorem~1.1]{Wang}. Since $\lambda$ is not very even, $\cO$ has a unique
preimage $\cO_0$ by Lemma~\ref{lem:split}, and $\cO_0\in\sO(\sigma)$ by
\eqref{eq:irred-res}. All these preimages lie in the unique stable $G_0$-orbit
with partition $\lambda$, \rev{and every orbit in $\sO(\sigma)$ has partition
at most $\lambda$; hence this stable orbit is the Kawanaka wavefront orbit, that
is, $\cO_0^{\st}=\cO_{\sigma}^{\st}$.} This stable orbit is fixed by $\Ad(c)$ by
Lemma~\ref{lem:conj}.
\end{proof}

We now describe $\sO(\sigma)^{\mx}$ and its rational closure. Let
$\cO_1\in\sO(\sigma)$. By Lemma~\ref{lem:pr-WF}
and \cite[Theorem~1.1]{Wang}, there is an orbit $\cO\in\sO(\pi)^{\mx}$ such that
$\pr(\cO_1)\leqslant_G\cO$. By Proposition~\ref{prop:irred-max}, $\cO$ has a
unique preimage $\cO_2$, and $\cO_2\in\sO(\sigma)$. Since the partition of
$\cO_2$ is not very even, {by Lemma~\ref{lem:order}(ii), we have}
$\cO_1\leqslant_{G_0}\cO_2$. Conversely, if $\cO_1\leqslant_{G_0}\cO_2$ for
\rev{some $\cO_2\in{\pr^{-1}(\sO(\pi)^{\mx})}$, then $\pr(\cO_1)\leqslant_G\pr(\cO_2)$, so}
$\pr(\cO_1)\in\sO(\pi)$ by \cite[Theorem~1.1]{Wang}, \rev{and}
$\cO_1\in\sO(\sigma)$ by \eqref{eq:irred-res}. \rev{Hence
$\sO(\sigma)=\overline{{\pr^{-1}(\sO(\pi)^{\mx})}}$. Distinct elements of ${\pr^{-1}(\sO(\pi)^{\mx})}$ are
incomparable, since by Lemma~\ref{lem:order}(i) a strict inequality between
them would give a strict inequality between two elements of
$\sO(\pi)^{\mx}$. Therefore $\sO(\sigma)^{\mx}={\pr^{-1}(\sO(\pi)^{\mx})}$. Finally, every
element of ${\pr^{-1}(\sO(\pi)^{\mx})}$ lies in $\cO_\sigma^{\st}$ by
Proposition~\ref{prop:irred-max}; conversely, if $\cO_0\in\sO(\sigma)$ lies in
$\cO_\sigma^{\st}$, then $\pr(\cO_0)\in\sO(\pi)$ has partition $\lambda$, so
$\pr(\cO_0)\in\sO(\pi)^{\mx,\st}=\sO(\pi)^{\mx}$. Thus
$\sO(\sigma)^{\mx,\st}={\pr^{-1}(\sO(\pi)^{\mx})}$. This proves the following theorem.}

\begin{theorem}\label{thm:irred}
Let $\pi\in\Irr(G^F)$ and suppose that $\sigma=\ResG\pi$ is irreducible. Then
\[
\sO(\sigma)^{\mx}=\sO(\sigma)^{\mx,\st}=\pr^{-1}\bigl(\sO(\pi)^{\mx}\bigr),
\qquad
\sO(\sigma)=\overline{\sO(\sigma)^{\mx}},
\]
and $\pr$ restricts to a bijection
$\sO(\sigma)^{\mx}\xrightarrow{\sim}\sO(\pi)^{\mx}$. Moreover,
$\dim\Wh_{\cO_0}(\sigma)=\dim\Wh_{\pr(\cO_0)}(\pi)$ for every rational
$G_0^F$-orbit $\cO_0$.
\end{theorem}

% ===============================================================
\section{Maximal orbits and Whittaker dimensions}\label{sec:max}
% ===============================================================

\subsection{\texorpdfstring{\rev{Very even maximal orbits}}{Very even maximal orbits}}

\begin{proposition}[Separation at very even maximal orbits]\label{prop:separation}
Let $\dim V=2n$, let $\pi\in\Irr(G^F)$ satisfy $\pi\otimes\det\simeq\pi$, and
write $\ResG\pi=\sigma^{\I}\oplus\sigma^{\II}$. Assume that the partition
$\lambda$ of $\cO_{\pi}^{\st}$ is very even, let $\cO_\lambda$ be the
rational $G^F$-orbit with partition $\lambda$, and write
$\pr^{-1}(\{\cO_\lambda\})=\{\cO_\lambda^{\I},\cO_\lambda^{\II}\}$. After
relabeling $\sigma^{\I}$ and $\sigma^{\II}$ if necessary,
\begin{equation}\label{eq:separation}
\begin{gathered}
\dim\Wh_{\cO_\lambda^{\I}}(\sigma^{\I})=\dim\Wh_{\cO_\lambda^{\II}}(\sigma^{\II})
=\dim\Wh_{\cO_\lambda}(\pi),\\
\Wh_{\cO_\lambda^{\II}}(\sigma^{\I})=\Wh_{\cO_\lambda^{\I}}(\sigma^{\II})=0,
\end{gathered}
\end{equation}
and
\[
\sO(\sigma^{\I})^{\mx}=\cO_\lambda^{\I},\qquad
\sO(\sigma^{\II})^{\mx}=\cO_\lambda^{\II}.
\]
\end{proposition}

\begin{proof}
By Proposition~\ref{prop:restriction},
\[
\dim\Wh_{\cO_\lambda^{\I}}(\sigma^{\I})+\dim\Wh_{\cO_\lambda^{\I}}(\sigma^{\II})
=\dim\Wh_{\cO_\lambda}(\pi)>0,
\]
so, after relabeling, $\Wh_{\cO_\lambda^{\I}}(\sigma^{\I})\neq0$. By
Lemma~\ref{lem:pr-WF}, $\pr(\cO_0)\in\sO(\pi)$ for every
$\cO_0\in\sO(\sigma^{\I})$; since $\lambda$ is {the largest partition of orbits occurring}
in $\sO(\pi)$, the stable orbit $(\cO_\lambda^{\I})^{\st}$ is maximal among the
stable orbits meeting $\sO(\sigma^{\I})$.

Since $\sigma^{\II}\simeq(\sigma^{\I})^c$ and
$\Ad(c)\cO_\lambda^{\I}=\cO_\lambda^{\II}$, Proposition~\ref{prop:conj} gives
\[
\Wh_{\cO_\lambda^{\II}}(\sigma^{\II})\simeq\Wh_{\cO_\lambda^{\I}}(\sigma^{\I})\neq0,
\qquad
\Wh_{\cO_\lambda^{\I}}(\sigma^{\II})\simeq\Wh_{\cO_\lambda^{\II}}(\sigma^{\I}).
\]
If $\Wh_{\cO_\lambda^{\II}}(\sigma^{\I})$ were nonzero, the stable orbit
$(\cO_\lambda^{\II})^{\st}\neq(\cO_\lambda^{\I})^{\st}$ would also be maximal
among the stable orbits meeting $\sO(\sigma^{\I})$, contradicting the
uniqueness of the Kawanaka wavefront orbit. Hence
$\Wh_{\cO_\lambda^{\II}}(\sigma^{\I})=0$, {which also implies that}
$\Wh_{\cO_\lambda^{\I}}(\sigma^{\II})=0$. In particular,
$\cO_{\sigma^{\I}}^{\st}=(\cO_\lambda^{\I})^{\st}$, and the dimension
equalities in \eqref{eq:separation} follow from
Proposition~\ref{prop:restriction}.

Let $\cO_1\in\sO(\sigma^{\I})$. By Lemma~\ref{lem:pr-WF} and
\cite[Theorem~1.1]{Wang}, $\pr(\cO_1)\leqslant_G\cO_\lambda$, and by
{\cite[Proposition~15.2]{Tay}}, $\cO_1^{\st}\subset\overline{(\cO_\lambda^{\I})^{\st}}$.
Thus $\cO_1\leqslant_{G_0}\cO_\lambda^{\I}$, and
$\sO(\sigma^{\I})^{\mx}=\cO_\lambda^{\I}$. The same argument applies to $\sigma^{\II}$.
\end{proof}

\begin{corollary}\label{cor:very-even-one}
In the situation of Proposition~\ref{prop:separation},
\[
\dim\Wh_{\cO_\lambda}(\pi)=\dim\Wh_{\cO_\lambda^{\I}}(\sigma^{\I})
=\dim\Wh_{\cO_\lambda^{\II}}(\sigma^{\II})=1.
\]
\end{corollary}

\begin{proof}
The centralizer in the full orthogonal group of {a nilpotent element whose partition is very even} is connected, so
$|A_{\mathrm{O}}(\cO^{\st}_\lambda)|=1$, and there is only one rational
$G^F$-orbit with partition $\lambda$. {By \cite[Theorem~1.9]{Wang}, we have} $\dim\Wh_{\cO_\lambda}(\pi)=1/|\bA(\cO^{\st*})|$. The left-hand side is a positive integer, and the denominator is the order of a
finite group. Hence both are equal to $1$.
\end{proof}

\subsection{Maximal orbits that are not very even}\label{subsec:not-very-even}

Let $\pi\in\Irr(G^F)$ with $\ResG\pi=\sigma^{\I}\oplus\sigma^{\II}$, and assume
that the partition $\lambda$ of $\cO_{\pi}^{\st}$ is not very even. For
$\cO\in\sO(\pi)^{\mx}$ let $\cO_0$ be its unique preimage under $\pr$. By
Corollary~\ref{cor:half}, $\Wh_{\cO_0}(\sigma^{\I})\neq0$ and
$\Wh_{\cO_0}(\sigma^{\II})\neq0$. By Lemma~\ref{lem:pr-WF} and
\cite[Theorem~1.1]{Wang}, every partition occurring in either wavefront set is dominated by $\lambda$; hence, writing $\cO^{\st}_\lambda$ for the unique stable
$G_0$-orbit with partition $\lambda$,
\[
\cO_0^{\st}=\cO_{\sigma^{\I}}^{\st}=\cO_{\sigma^{\II}}^{\st}
=\cO^{\st}_\lambda.
\]

Fix $\sigma\in\{\sigma^{\I},\sigma^{\II}\}$ and let $\cO_1\in\sO(\sigma)$. By
Lemma~\ref{lem:pr-WF} and \cite[Theorem~1.1]{Wang}, there is an orbit
$\cO\in\sO(\pi)^{\mx}$ such that $\pr(\cO_1)\leqslant_G\cO$. Let $\cO_0$ be the
preimage of $\cO$. Then $\cO_0\in\sO(\sigma)$ by Corollary~\ref{cor:half}, and
$\cO_1\leqslant_{G_0}\cO_0$ by Lemma~\ref{lem:order}, since the partition of
$\cO_0$ is not very even. Conversely, let $\cO\in\sO(\pi)^{\mx}$ with preimage
$\cO_0$. If $\cO_0<_{G_0}\cO_1$ for some $\cO_1\in\sO(\sigma)$, then
$\cO<_G\pr(\cO_1)$ by Lemma~\ref{lem:order}, while $\pr(\cO_1)\in\sO(\pi)$ by
Lemma~\ref{lem:pr-WF}; this contradicts the maximality of $\cO$. \rev{Finally,
if $\cO_1\in\sO(\sigma)$ satisfies $\cO_1^{\st}=\cO^{\st}_\lambda$, then
$\pr(\cO_1)\in\sO(\pi)$ has partition $\lambda$, so
$\pr(\cO_1)\in\sO(\pi)^{\mx,\st}=\sO(\pi)^{\mx}$.} Therefore
\[
\sO(\sigma)^{\mx}=\sO(\sigma)^{\mx,\st}=\pr^{-1}\bigl(\sO(\pi)^{\mx}\bigr).
\]
Together with Theorem~\ref{thm:irred} and
Proposition~\ref{prop:separation}, this proves Theorem~\ref{thm:A}, which we
restate here.

\begin{theorem}\label{thm:max}
For every $\sigma\in\Irr(G_0^F)$ we have
$\sO(\sigma)^{\mx}=\sO(\sigma)^{\mx,\st}$. If $\sigma$ occurs in $\ResG\pi$ with
$\pi\in\Irr(G^F)$, then
\[
\sO(\sigma)^{\mx}
=\bigcup_{\substack{\cO_0\in\Nil^F/\Ad(G_0^F)\\
\pr(\cO_0)\in\sO(\pi)^{\mx},\ \cO_0^{\st}=\cO_{\sigma}^{\st}}}\cO_0.
\]
Let $\lambda$ be the partition of $\cO_{\pi}^{\st}$. If $\lambda$ is not
very even, $\pr$ induces a bijection
$\sO(\sigma)^{\mx}\xrightarrow{\sim}\sO(\pi)^{\mx}$. If $\lambda$ is very even,
write $\ResG\pi=\sigma^{\I}\oplus\sigma^{\II}$ and
$\pr^{-1}(\{\cO_\lambda\})=\{\cO_\lambda^{\I},\cO_\lambda^{\II}\}$; after
relabeling the constituents, $\sO(\sigma^{\I})^{\mx}=\cO_\lambda^{\I}$ and
$\sO(\sigma^{\II})^{\mx}=\cO_\lambda^{\II}$.
\end{theorem}

\begin{corollary}\label{cor:inclusion}
Every $\sigma\in\Irr(G_0^F)$ satisfies
$\sO(\sigma)\subset\overline{\sO(\sigma)^{\mx}}$.
\end{corollary}

\begin{proof}
Since there are finitely many rational orbits, every $\cO_0\in\sO(\sigma)$
satisfies $\cO_0\leqslant_{G_0}\cO_1$ for some $\cO_1\in\sO(\sigma)^{\mx}$.
\end{proof}

\subsection{Proof of Theorem~\ref{thm:C}}

\rev{Let $\sigma$, $\pi$ and $\cO^{\st}=\cO_\sigma^{\st}$ be as in
Theorem~\ref{thm:C}. By Propositions~\ref{prop:irred-max}
and~\ref{prop:separation} and Section~\ref{subsec:not-very-even}, the
partition of $\cO^{\st}$ is that of $\cO_\pi^{\st}$.}
Suppose first that the partition of $\cO^{\st}$ is not very even. By
\eqref{eq:det} there is an exact sequence
\[
1\longrightarrow A_{\mathrm{SO}}(\cO^{\st})\longrightarrow
A_{\mathrm{O}}(\cO^{\st})\xrightarrow{\ \det\ }\{\pm1\}\longrightarrow1,
\]
so $|A_{\mathrm{O}}(\cO^{\st})|=2|A_{\mathrm{SO}}(\cO^{\st})|$. By
Theorem~\ref{thm:max}, $|\sO(\sigma)^{\mx}|=|\sO(\pi)^{\mx}|$. If $\dim V$ is
odd, the restriction is irreducible \rev{(Section~\ref{subsec:irred})}, and substitution of these identities into the
odd case of \eqref{eq:Wh-dim-O} together with \eqref{eq:irred-res} gives the second
case of \eqref{eq:Wh-dim-SO}. If $\dim V$ is even and the restriction is
irreducible, the same substitution into the even case of \eqref{eq:Wh-dim-O}
gives the first case. If the restriction is reducible, \eqref{eq:half} divides
$\dim\Wh_{\cO}(\pi)$ by two, yielding the second case again.

If the partition is very even, the restriction is reducible, both component
groups are trivial, and $\sO(\pi)^{\mx}$ and $\sO(\sigma)^{\mx}$ each consist of
one rational orbit. Proposition~\ref{prop:separation} and \eqref{eq:Wh-dim-O}
give the second case of \eqref{eq:Wh-dim-SO}, and $\dim\Wh_{\cO_0}(\sigma)=1$ by
Corollary~\ref{cor:very-even-one}.

\rev{In each case, the right-hand side of \eqref{eq:Wh-dim-SO} is independent
of $\cO_0$. Finally, $\dim\Wh_{\cO}(\pi)$ is a power of two by
\cite[Theorem~1.9]{Wang}, and passing to $\sigma$ either preserves this number
or halves it. As $\dim\Wh_{\cO_0}(\sigma)$ is a positive integer, it is also a
power of two.} \qed

\begin{remark}
\rev{The formula in Theorem~\ref{thm:C} is independent of the choice of $\pi$. If $\sigma$
extends to $G^F$, its two extensions $\pi$ and $\pi\otimes\det$ have the same
wavefront set and Whittaker multiplicities
(Proposition~\ref{prop:restriction}), so \eqref{eq:Wh-dim-O} gives the same
value of $|\bA(\cO^{\st*})|$ for both. Otherwise $\IndG\sigma$ is the only
irreducible representation of $G^F$ whose restriction contains $\sigma$.}
\end{remark}

% ===============================================================
\section{Nonvanishing at smaller very even orbits}\label{sec:closure}
% ===============================================================

We prove that the generalized Whittaker models of $\sigma$ are nonzero at both
rational $G_0^F$-orbits attached to a very even partition lying strictly below
an orbit in $\sO(\sigma)^{\mx}$. The proof uses an intermediate orbit obtained
by replacing two equal even parts $a,a$ by $a+1,a-1$, together with a comparison
of generalized Gelfand--Graev representations. All algebraic groups and gradings below are defined over $\kk$.
Their Lie algebras and weight spaces are taken over $\kb$.

\subsection{An intermediate orbit}

\begin{lemma}\label{lem:intermediate}
Let $D_\lambda$ be the signed Young diagram of a rational $G^F$-orbit which corresponds to a very even partition $\lambda$,
and let $D$ be a signed Young diagram for the same quadratic space with
$D_\lambda<_GD$. Then there exist an even part $a$ of $\lambda$ and a signed
Young diagram $D_\nu$, obtained from $D_\lambda$ by replacing two parts equal to
$a$ by $a+1$ and $a-1$, such that
\[
D_\lambda<_GD_\nu\leqslant_GD .
\]
The partition $\nu$ is not very even.
\end{lemma}

\begin{proof}
Since $D_\lambda<_G D$, we have $\lambda<\mu$ in the dominance order,
where $\mu$ is the partition of $D$.
Let $i_0$ be the smallest index such that
$\lambda_{i_0}<\mu_{i_0}$, and set $a=\lambda_{i_0}$.
Then $\lambda_i=\mu_i$ for all $i<i_0$, and hence
\[
C_a(D_\lambda)-C_a(D)
=\sum_{i\geq i_0}\max\{\mu_i-a,0\}
\geq \mu_{i_0}-a>0.
\]

Choose $b=\sgn_a(D)$ if
$C_a(D_\lambda)-C_a(D)=1$, and choose either sign otherwise.
Let $D_\nu$ be the diagram obtained from $D_\lambda$ by
replacing two parts $a$ with the signed parts
$(a+1,b)$ and $(a-1,\varepsilon(-1)b)$,
leaving all other parts unchanged. Then
\begin{equation}\label{eq:C-nu}
C_r(D_\nu)=
\begin{cases}
C_a(D_\lambda)-1,&r=a,\\
C_r(D_\lambda),&r\neq a.
\end{cases}
\end{equation}
The two odd parts occur once each, and all remaining parts
are even with even multiplicity. Their signs have product
$\varepsilon(-1)$, which cancels the factor
$\varepsilon(-1)^{1\cdot1}$ arising from the pair
$(a+1,a-1)$. Hence
$D_\nu$ is a signed Young diagram for $V$, and
\begin{equation}\label{eq:sgn-nu}
\sgn_{2s}(D_\nu)=
\begin{cases}
b,&2s=a,\\
+,&2s\neq a.
\end{cases}
\end{equation}
Since $\sgn_{2s}(D_\lambda)=+$ for all $s$, equations \eqref{eq:C-nu} and
\eqref{eq:sgn-nu} together with $D_\lambda\leqslant_GD$ verify the conditions of
\eqref{eq:full-order} for $D_\nu\leqslant_GD$; they also show
$D_\lambda<_GD_\nu$.
\end{proof}

\subsection{Comparison of generalized Gelfand--Graev representations}

The following comparison, which uses finite Heisenberg groups and induction in
stages, is a finite-group version of the comparison of models underlying
\cite[Section~3.2]{GGS}.

\begin{lemma}\label{lem:comparison}
Let $\cO_\lambda^{\I},\cO_\lambda^{\II}$ be the two rational $G_0^F$-orbits of a
very even partition $\lambda$, let $a$ be an even part of $\lambda$, and let
$\cO_\nu$ be a rational $G_0^F$-orbit whose signed Young diagram is obtained by
replacing two parts $a$ by $(a+1,b)$ and $(a-1,\varepsilon(-1)b)$, all other
parts being unchanged. For $\rev{\iota}\in\{\I,\II\}$ there is a surjective
$G_0^F$-homomorphism
\begin{equation}\label{eq:surj}
\Gamma^{G_0}_{\cO_\lambda^{\rev{\iota}}}\twoheadrightarrow\Gamma^{G_0}_{\cO_\nu}.
\end{equation}
Consequently
$\Wh_{\cO_\nu}(\tau)\hookrightarrow\Wh_{\cO_\lambda^{\rev{\iota}}}(\tau)$ for every
representation $\tau$ of $G_0^F$.
\end{lemma}

\begin{proof}
\emph{Step 1: the $\sltwo(\kk)$-triple for $\cO_\nu$.}
Fix $X\in\cO_\lambda^{\rev{\iota}}$ and an $\sltwo(\kk)$-triple $\{X,H,Y\}$, and use the
decomposition \eqref{eq:isotypic}. Choose a nondegenerate two-dimensional
subspace $W_a^0\subset W_a$; recall that the form on $W_a$ is skew-symmetric since $a$ is even.
Choose an $\sltwo(\kk)$-triple $\{X_1,H_1,Y_1\}$ in $\spl(W_a^0)$. Extend
each of its members by zero on $(W_a^0)^\perp\subset W_a$, let it act on
$L_a\otimes W_a$ through the second factor, and let it act by zero on
$L_d\otimes W_d$ for $d\neq a$.
The resulting triple in $\g^F$ commutes with $\{X,H,Y\}$, and therefore
$\{X+X_1,H+H_1,Y+Y_1\}$ is a rational $\sltwo$-triple. Since
$L_a\otimes L_2\simeq L_{a+1}\oplus L_{a-1}$, the element $X+X_1$ has partition
$\nu$.

\finalrev{Both values of $b$ occur for suitable choices of the triple
$\{X_1,H_1,Y_1\}$. We choose this triple so that
$X+X_1\in\cO_\nu$.}

\emph{Step 2: a bigrading.}
\finalrev{Let $\gamma$ and $\gamma_1$ be the commuting cocharacters
associated with the two $\mathrm{SL}_2$-actions. For $i,j\in\mathbb Z$, put
\[
\g_{i,j}=\{T\in\g:
\Ad(\gamma(s)\gamma_1(t))T=s^it^jT
\text{ for all }s,t\in\kb^\times\}.
\]
Since $\lambda$ is very even, all weights of $\gamma$ on
$V_{\kb}$ are odd, so $\g_{i,j}=0$ for odd $i$.
The weights of $\gamma_1$ on $V_{\kb}$ are $-1,0,1$,
so $\g_{i,j}=0$ for $|j|>2$.} Put
\[
\mathfrak{n}=\bigoplus_{i\geqslant2,\,j}\g_{i,j},\qquad
\mathfrak{a}=\g_{2,-2},\qquad
\mathfrak{b}=\g_{0,2},\qquad
\mathfrak{n}_0=\bigoplus_{\substack{i\geqslant2,\,j\\(i,j)\neq(2,-2)}}\g_{i,j},
\]
so that $\mathfrak{n}=\mathfrak{n}_0\oplus\mathfrak{a}$. Define linear forms
$\varphi(T)=\langle Y,T\rangle$ and $\varphi_1(T)=\langle Y_1,T\rangle$, and
consider the pairing
\begin{equation}\label{eq:pairing}
\mathfrak{a}\times\mathfrak{b}\longrightarrow\kb,\qquad
(A,B)\longmapsto\varphi([A,B]).
\end{equation}

By invariance of $\langle\ ,\ \rangle$, we have
\[
\varphi([A,B])
=\langle [Y,A],B\rangle
=-\langle A,[Y,B]\rangle.
\]
Since $\langle\ ,\ \rangle$ induces a nondegenerate
pairing between $\g_{i,j}$ and $\g_{-i,-j}$,
the left and right kernels of \eqref{eq:pairing}
are respectively
\[
\ker(\ad Y)\cap\g_{2,-2}
\quad\text{and}\quad
\ker(\ad Y)\cap\g_{0,2}.
\]

We first determine these kernels using the
$\sltwo$-triple $\{X,H,Y\}$.
On each irreducible $\sltwo$-module,
$\ker Y$ is the lowest {$\gamma$-weight} space.
In particular, it contains no vectors of
positive {$\gamma$-weight}, and its {$\gamma$-weight}
zero subspace consists of vectors fixed
by the entire triple.
It follows that
\[
\ker(\ad Y)\cap\g_{2,-2}=0
\]
and
\[
\ker(\ad Y)\cap\g_{0,2}
=
\Lie C_{G_0}(\{X,H,Y\})\cap\g_{0,2}.
\]

We now compute the latter space.
Since $\lambda$ is very even, the isotypic
decomposition \eqref{eq:isotypic} gives
\[
\Lie C_{G_0}(\{X,H,Y\})
\simeq
\bigoplus_d \spl(W_{d,\kb}).
\]
Recall that {$\gamma_1$} acts trivially on
$W_{d,\kb}$ for $d\neq a$.
Write
\[
W_{a,\kb}=W_{a,\kb}^0\oplus U_a,
\qquad
U_a=(W_{a,\kb}^0)^\perp,
\]
where $W_{a,\kb}^0=W_a^0\otimes_{\kk}\kb$.
The weights of {$\gamma_1$} on $W_{a,\kb}^0$
are $1$ and $-1$, while {$\gamma_1$} acts
trivially on $U_a$.

Using the identification
\(
\spl(W_{a,\kb})
\simeq \operatorname{Sym}^2(W_{a,\kb}),
\)
we obtain the {$\gamma_1$}-stable decomposition
\[
\operatorname{Sym}^2(W_{a,\kb})
=
\operatorname{Sym}^2(W_{a,\kb}^0)
\oplus (W_{a,\kb}^0\otimes U_a)
\oplus \operatorname{Sym}^2(U_a).
\]
The {$\gamma_1$}-weights of these three summands
are respectively
\[
\{2,0,-2\},\qquad
\{1,-1\},\qquad
\{0\}.
\]
Thus the {$\gamma_1$-weight $2$} subspace comes 
from $\operatorname{Sym}^2(W_{a,\kb}^0)$.
Since
\(
\spl(W_{a,\kb}^0)\simeq\sltwo(\kb),
\)
its {$\gamma_1$-weight $2$} subspace is precisely
$\kb X_1$. Therefore
\(
\ker(\ad Y)\cap\g_{0,2}
=
\kb X_1.
\)

By the definition of the $G$-invariant bilinear form
$\langle\ ,\ \rangle$ on $\g$ fixed in
Section~\ref{sec:models}, namely
$\langle Z,Z'\rangle=\operatorname{tr}_{V_{\kb}}(ZZ')$,
and since $X_1,Y_1$ act on
$L_{a,\kb}\otimes_{\kb}W_{a,\kb}^0$ through the second factor
and vanish on the other summands, we have
\[
\varphi_1(X_1)=\langle Y_1,X_1\rangle
=a\,\operatorname{tr}_{W_{a,\kb}^0}(Y_1X_1)
=a\neq0,
\]
where the trace on $W_{a,\kb}^0$ equals $1$ by the
standard two-dimensional $\sltwo$-representation,
and the last inequality follows from $a<p$.
Consequently,
\[
\mathfrak{b}'=\mathfrak{b}\cap\ker\varphi_1
\]
is a complement to $\kb X_1$ in
$\mathfrak{b}$.
Since the left kernel of \eqref{eq:pairing}
is zero and its right kernel is $\kb X_1$,
the restriction to
$\mathfrak{a}\times\mathfrak{b}'$
is a perfect pairing.
All these constructions are defined
over $\kk$, so the induced pairing on
$\mathfrak{a}^F\times(\mathfrak{b}')^F$
is also perfect.

\emph{Step 3: the Heisenberg comparison.}
Set
\[
\mathfrak u=\mathfrak n\oplus\mathfrak b',
\qquad
\mathfrak r=\mathfrak n_0\oplus\mathfrak b'.
\]
Since $\g_{i,j}=0$ for $|j|>2$,
we have
\[
[\mathfrak b',\mathfrak b']=0,
\qquad
[\mathfrak n,\mathfrak b']\subset\mathfrak n,
\qquad
[\mathfrak n_0,\mathfrak b']\subset\mathfrak n_0.
\]
Thus $\mathfrak u$ and $\mathfrak r$ are Lie subalgebras.
{Moreover, $\g_{i,j}$ has weight $2i+j$ under the cocharacter
$\widetilde{\gamma}(t)=\gamma(t^2)\gamma_1(t)$, and this weight is
positive for every $\g_{i,j}$ occurring in $\mathfrak u$.}
This positive grading shows that
$\mathfrak u$, and hence $\mathfrak n$ and $\mathfrak r$,
are nilpotent.

Consider the skew-symmetric form
\[
B_\varphi(T,T')=\varphi([T,T'])
\]
on $\mathfrak u$.
Since $\varphi$ is supported on $\g_{2,0}$, the only
bidegrees that can contribute to $B_\varphi$ are
$(2,-2)$ and $(0,2)$. Thus $\mathfrak n_0$ is
contained in the radical of $B_\varphi$.
By Step~2, the pairing between
$\mathfrak a=\g_{2,-2}$ and $\mathfrak b'$
is perfect. Hence the radical of $B_\varphi$ is
exactly $\mathfrak n_0$, and
\[
\mathfrak u/\mathfrak n_0
=\mathfrak a\oplus\mathfrak b'
\]
is a symplectic space, with $\mathfrak a$ and
$\mathfrak b'$ as complementary Lagrangian subspaces.
In particular,
\[
\varphi([\mathfrak n,\mathfrak n])
=\varphi([\mathfrak r,\mathfrak r])=0.
\]

We next pass from Lie algebras to groups.
Since $\lambda$ is very even, $4\mid\dim V$.
Our characteristic hypothesis implies $p>\dim V$:
for $\dim V=4$ this follows from $p>3$, while for
$\dim V\geq8$ it follows from
$p>3(h_G-1)$ and $h_G=\dim V-2$.
The positive grading defined by {$\widetilde{\gamma}$} makes
$\mathfrak u$ strictly upper triangular on $V$.
Thus every $T\in\mathfrak u$ satisfies
$T^{\dim V}=0$.
Since $p>\dim V$, the truncated exponential and
logarithm are well-defined, and the
Baker--Campbell--Hausdorff formula applies.
Consequently, the Lie subalgebras
$\mathfrak u$, $\mathfrak n$, and $\mathfrak r$
give rise to unipotent subgroups of $G_0$. Define
\[
U=\exp(\mathfrak u),\qquad
N=\exp(\mathfrak n),\qquad
R=\exp(\mathfrak r).
\]
These are unipotent subgroups of $G_0$ defined
over $\kk$, with
$U^F=\exp(\mathfrak u^F)$, and similarly for $N$ and $R$.

By the definition of the Cayley map in
Section~\ref{sec:models}, we have
\[
\Cay(\exp T)=\frac12 T+
\text{terms of odd degree at least $3$}.
\]
Since $\varphi$ vanishes on
$[\mathfrak n,\mathfrak n]$ and
$[\mathfrak r,\mathfrak r]$, the formula
\[
\chi_\varphi(\exp T)=\psi(\frac12\varphi(T))
\]
defines characters of $N^F$ and $R^F$.
Every $(i,j)$ in $\mathfrak u$ satisfies
$2i+j\geq2$, whereas $Y$ and $Y_1$ have
{$\widetilde{\gamma}$}-weights $-4$ and $-2$, respectively.
Thus the higher-order terms of the Cayley map
pair trivially with $Y$ and $Y_1$.
In particular, $\chi_\varphi$ on $N^F$ is the
character defining
$\Gamma^{G_0}_{\cO_\lambda^{\rev{\iota}}}$
in Section~\ref{sec:models}, {since $\gamma$ has no weight $1$ on $\g$}
and $G_{\geqslant1.5}^F=N^F$.

To apply the Stone--von Neumann theorem, put
\[
\mathfrak k_0=\mathfrak n_0\cap\ker\varphi.
\]
Note that
$[\mathfrak u,\mathfrak u]\subset\mathfrak n_0$.
Since $\mathfrak n_0$ is the radical of $B_\varphi$,
we also have
$[\mathfrak u,\mathfrak n_0]\subset\mathfrak k_0$.
Hence $\mathfrak k_0$ is an ideal of $\mathfrak u$,
and $\mathfrak u/\mathfrak k_0$ is two-step nilpotent.
Moreover, $\varphi(\mathfrak n_0)=\kb$, since
$Y\neq0$ lies in $\g_{-2,0}$ and the pairing
between $\g_{-2,0}$ and
$\g_{2,0}\subset\mathfrak n_0$ is nondegenerate.
If $\mathfrak a\neq0$, it follows that
\[
U^F/\exp(\mathfrak k_0^F)
\]
is a Heisenberg group, with center
$\exp(\mathfrak n_0^F)/\exp(\mathfrak k_0^F)$
and symplectic quotient
$\mathfrak a^F\oplus(\mathfrak b')^F$.
The subgroups $N^F$ and $R^F$ correspond to
the two complementary Lagrangian subspaces,
and their characters agree on the center.
Therefore, the Stone--von Neumann theorem gives
\[
\Ind_{N^F}^{U^F}\chi_\varphi
\simeq
\Ind_{R^F}^{U^F}\chi_\varphi.
\]
If $\mathfrak a=0$, then $\mathfrak b'=0$ and
$N^F=R^F=U^F$, so the same conclusion holds.
Finally, induction in stages yields
\begin{equation}\label{eq:NR}
\Gamma^{G_0}_{\cO_\lambda^{\rev{\iota}}}
=\Ind_{N^F}^{G_0^F}\chi_\varphi
\simeq
\Ind_{R^F}^{G_0^F}\chi_\varphi.
\end{equation}

\emph{Step 4: passage to $\cO_\nu$.}
Set $H_\nu=H+H_1$ {and $\gamma_\nu(t)=\gamma(t)\gamma_1(t)$}.
Using the simultaneous weight spaces $\g_{i,j}$, define
\[
\g_k({\gamma_\nu})=\bigoplus_{i+j=k}\g_{i,j},
\qquad
\g_{\geqslant k}({\gamma_\nu})=\bigoplus_{r\geqslant k}\g_r({\gamma_\nu}),
\]
and put $\mathfrak u_\nu=\g_{\geqslant1}({\gamma_\nu})$.

We first show that $\mathfrak r=\mathfrak{n}_0\oplus\mathfrak{b}'\subset\mathfrak u_\nu$.
Every $\g_{i,j}$ occurring in $\mathfrak n_0$
satisfies $i\geq2$, $j\geq-2$, and $(i,j)\neq(2,-2)$.
Hence $i+j\geq1$. Moreover,
$\mathfrak b'\subset\g_{0,2}$ has {$\gamma_\nu$}-weight $2$.
Thus $\mathfrak r\subset\mathfrak u_\nu$.
By the weight decomposition and the definition of
$\mathfrak b'$, we also have
$\varphi_1|_{\mathfrak r}=0$.
Together with Step~3, this shows that $\mathfrak r$
is isotropic for the skew-symmetric form
\[
(T,T')\longmapsto(\varphi+\varphi_1)([T,T']).
\]

Consider the quotient
$
\mathfrak v_\nu
=\mathfrak u_\nu/\g_{\geqslant2}({\gamma_\nu}).
$
The triple $\{X+X_1,H_\nu,Y+Y_1\}$ defines a
nondegenerate alternating form on $\mathfrak v_\nu$ by
\[
(\bar T,\bar T')\longmapsto(\varphi+\varphi_1)([T,T']).
\]
The image $\overline{\mathfrak r}$ of $\mathfrak r$
in $\mathfrak v_\nu$ is isotropic.
Choose an $F$-stable Lagrangian subspace
$\mathfrak l_\nu$ containing $\overline{\mathfrak r}$,
and let $\mathfrak q$ be its preimage in
$\mathfrak u_\nu$. Then
$\mathfrak r\subset\mathfrak q$, and
$\mathfrak q$ is the Lie algebra used in
Section~\ref{sec:models} to define the generalized
Gelfand--Graev representation associated with $\cO_\nu$.

Put $Q=\exp(\mathfrak q)$.
Since $\mathfrak q\subset\g_{\geqslant1}({\gamma_\nu})$
and $Y+Y_1$ has {$\gamma_\nu$}-weight $-2$,
the higher-order terms of the fixed Cayley map
pair trivially with $Y+Y_1$.
Thus the character $\chi_{\varphi+\varphi_1}$ of $Q^F$
is the one defined in Section~\ref{sec:models}.
Since $\varphi_1|_{\mathfrak r}=0$, we have
\[
\chi_{\varphi+\varphi_1}|_{R^F}=\chi_\varphi.
\]
By Frobenius reciprocity and semisimplicity,
there is a surjection
\[
\Ind_{R^F}^{Q^F}\chi_\varphi
\twoheadrightarrow\chi_{\varphi+\varphi_1}.
\]
Induction in stages therefore gives
\[
\Ind_{R^F}^{G_0^F}\chi_\varphi
\twoheadrightarrow
\Ind_{Q^F}^{G_0^F}\chi_{\varphi+\varphi_1}
=\Gamma^{G_0}_{\cO_\nu}.
\]
Combining this with \eqref{eq:NR} proves
\eqref{eq:surj}. Applying
$\Hom_{G_0^F}(-,\tau)$ and Frobenius reciprocity
gives the desired injection of Whittaker models.

The argument applies to both
$\rev{\iota}\in\{\I,\II\}$.
Since $\nu$ is not very even, the signed Young
diagram of $X+X_1$ determines a unique rational
$G_0^F$-orbit $\cO_\nu$.

\end{proof}

\subsection{Nonvanishing and rational closure}

\begin{theorem}\label{thm:nonvanishing}
Let $\pi\in\Irr(G^F)$ satisfy $\pi\otimes\det\simeq\pi$, and let $\sigma$ be
either irreducible constituent of $\ResG\pi$. Let $\lambda$ be a very even
partition such that $\cO_\lambda<_G\widetilde{\cO}$ for some
$\widetilde{\cO}\in\sO(\pi)^{\mx}$. Then
\[
\Wh_{\cO_\lambda^{\I}}(\sigma)\neq0\qquad\text{and}\qquad
\Wh_{\cO_\lambda^{\II}}(\sigma)\neq0.
\]
\end{theorem}

\begin{proof}
By Lemma~\ref{lem:intermediate}, there is a rational $G^F$-orbit
$\widetilde{\cO}_\nu$, whose partition $\nu$ is obtained from $\lambda$ by
replacing $a,a$ by $a+1,a-1$, such that
$\cO_\lambda<_G\widetilde{\cO}_\nu\leqslant_G\widetilde{\cO}$. Let $\cO_\nu$ be
the unique rational $G_0^F$-orbit with $\pr(\cO_\nu)=\widetilde{\cO}_\nu$. By
\cite[Theorem~1.1]{Wang} and Corollary~\ref{cor:half},
\[
\dim\Wh_{\cO_\nu}(\sigma)=\tfrac12\dim\Wh_{\widetilde{\cO}_\nu}(\pi)>0.
\]
For each $\rev{\iota}\in\{\I,\II\}$, by Lemma~\ref{lem:comparison}, we have
$\dim\Wh_{\cO_\lambda^{\rev{\iota}}}(\sigma)\geqslant\dim\Wh_{\cO_\nu}(\sigma)>0$.
\end{proof}

\begin{proof}[Proof of Theorem~\ref{thm:B}]
The inclusion $\sO(\sigma)\subset\overline{\sO(\sigma)^{\mx}}$ follows from
Corollary~\ref{cor:inclusion}. For the reverse inclusion, choose
$\pi\in\Irr(G^F)$ whose restriction contains $\sigma$, and let
$\cO_1\leqslant_{G_0}\cO_2$ with $\cO_2\in\sO(\sigma)^{\mx}$. If
$\cO_1=\cO_2$ there is nothing to prove, so assume $\cO_1<_{G_0}\cO_2$. By
Theorem~\ref{thm:max}, $\pr(\cO_2)\in\sO(\pi)^{\mx}$, and by
Lemma~\ref{lem:order}(i), $\pr(\cO_1)<_G\pr(\cO_2)$; hence
$\Wh_{\pr(\cO_1)}(\pi)\neq0$ by \cite[Theorem~1.1]{Wang}. If the restriction of
$\pi$ is irreducible, $\Wh_{\cO_1}(\sigma)\neq0$ by \eqref{eq:irred-res}. If it
is reducible and \rev{the partition of $\cO_1$} is not very even, the same
conclusion follows from
Corollary~\ref{cor:half}. If the restriction is reducible and
\rev{the partition of $\cO_1$} is very even,
we have
$\Wh_{\cO_1}(\sigma)\neq0$ by Theorem~\ref{thm:nonvanishing}.
\end{proof}

\section{Unipotent representations and families}\label{sec:unipotent}
% ===============================================================

\subsection{Families and restriction}\label{subsec:families}

For a stable nilpotent orbit $\cO^{\st}$, let $A_{\mathrm{O}}(\cO^{\st})$ and
$\bA_{\mathrm{O}}(\cO^{\st})$ denote the component group and canonical quotient for the full orthogonal
group, respectively. Let
$q_{\mathrm{O}}\colon A_{\mathrm{O}}(\cO^{\st})\to\bA_{\mathrm{O}}(\cO^{\st})$
be the quotient map. When $\cO^{\st}$ is fixed, we write $A_{\mathrm{O}}$ and
$\bA_{\mathrm{O}}$ for short.

Let $\cF$ be a family of unipotent representations of $G_0^F$ with stable
unipotent support $\cO^{\st}_{\cF}$. Let $\cF'$ be the family containing the
irreducible Alvis--Curtis duals $\sigma'$ of the members $\sigma\in\cF$.
By \cite[Section~2.1]{DM} and \cite[Lemma~14.15]{Tay},
\[
\cO_{\sigma}^{\st}=\cO^{\st}_{\cF'},\qquad
\cO_{\sigma'}^{\st}=\cO^{\st}_{\cF}.
\]
For a finite group $B$, put
\[
\mathcal{M}(B)=\{(x,\chi):x\in B,\ \chi\in\Irr(C_B(x))\}/B,
\]
where $B$ acts by simultaneous conjugation.

{We use symbols to classify unipotent representations, as in
\cite[Sections~2.2 and~3.3]{Pan} and \cite[Section~2.6]{Wang}.
Write
$\Lambda_\pi=\binom{A}{B}$ for the symbol attached to $\pi$ and
$\operatorname{def}(\Lambda_\pi)=|A|-|B|$ for its defect.
The rows are finite strictly decreasing sequences of nonnegative integers;
entries occurring in exactly one row are called singles, and a symbol is
called degenerate if its two rows coincide and nondegenerate otherwise.}

\smallskip
\noindent\emph{Even orthogonal groups.}
Suppose $\dim V=2n$, put $\epsilon=\disc(V)$, and let $\tF$ be a family
of unipotent representations of $G^F=\mathrm O_{2n}^{\epsilon}(\kk)$ with
stable unipotent support $\cO^{\st}_{\tF}$. {The symbols $\Lambda_\pi$ attached to the representations
$\pi\in\tF$ have the same multiset of entries, after choosing representatives
with the same total number of entries, and satisfy}
\[
\operatorname{def}({\Lambda_\pi})\equiv
\begin{cases}0\pmod4,&\epsilon=+,\\2\pmod4,&\epsilon=-.
\end{cases}
\]
See \cite[Sections~2.2 and~3.3]{Pan}.

{Fix $\pi\in\tF$ and suppose first that $\Lambda_\pi$ is nondegenerate,
with $2d$ singles, where $d\geqslant1$. By \cite[Section~2.6]{Wang},
\[
\bA_{\mathrm O}(\cO^{\st}_{\tF})\simeq(\mathbb Z/2\mathbb Z)^d.
\]}
\begingroup
Let $\Lambda_0$ be the special symbol of defect zero with the same multiset
of entries as $\Lambda_\pi$.
The representation corresponding to $\Lambda_0$ belongs to
$\mathrm O_{2n}^{+}(\kk)$, which differs from $G^F$ when $\epsilon=-$.
Following \cite[Section~2.6]{Wang}, order the singles as
$z_1>\cdots>z_{2d}$ and set $\delta_j=+$ if $z_j$ lies in the same row
in $\Lambda_\pi$ and $\Lambda_0$, and $\delta_j=-$ otherwise.
Since $\Lambda_0$ has defect zero, the defect condition gives
\[
\prod_{j=1}^{2d}\delta_j
=(-1)^{\operatorname{def}(\Lambda_\pi)/2}=\epsilon.
\]
\endgroup
{In the notation of \cite[Section~2.6]{Wang}, the parameter
$(x^+,x^-)\in\mathcal M(\bA_{\mathrm O}(\cO^{\st}_{\tF}))$ attached to $\pi$ has coordinates}
\[
x^+=(\epsilon_1,\ldots,\epsilon_d),\qquad
\epsilon_i=\prod_{j=1}^{2i}\delta_j,
\qquad
x^-=\left(\prod_{j=1}^{2i-1}\delta_j\right)_{i=1}^d.
\]
Then $\epsilon_d=\epsilon$, and the parameter set for $\tF$ is
\begin{equation}\label{eq:fixed-form-parameters}
{\mathcal M_\epsilon
:=\{(x^+,x^-)\in\mathcal M(\bA_{\mathrm O}(\cO^{\st}_{\tF})):
\epsilon_d=\epsilon\}.}
\end{equation}
{Conversely, every pair in this subset determines the symbol
of a unique representation in $\tF$.} Hence $|\tF|=2^{2d-1}$. {We write}
$\pi_{\cO^{\st}_{\tF},x^+,x^-}$ for the corresponding representation.
Let $-x^-$ denote the coordinatewise sign reversal of $x^-$ in the coordinates
of \cite[Section~2.6]{Wang}. Then
\[
\pi_{\cO^{\st}_{\tF},x^+,x^-}\otimes\det=\pi_{\cO^{\st}_{\tF},x^+,-x^-}.
\]
{Indeed, the symbol attached to $\pi\otimes\det$ is the transpose
of $\Lambda_\pi$ \cite[Section~3.3]{Pan}.} {Transposition preserves the defect modulo $4$ and changes
$\delta_j$ to $-\delta_j$ for every $1\leqslant j\leqslant2d$.} {Each product $\prod_{j=1}^{2i}\delta_j$ defining $x^+$ is
unchanged, whereas each product $\prod_{j=1}^{2i-1}\delta_j$ defining $x^-$
changes sign.}
In particular, this involution preserves $\mathcal M_\epsilon$.

\finalrev{Suppose that the two rows of $\Lambda_\pi$ are different.
Then $\ResG\pi$ is irreducible, and two such restrictions are
isomorphic if and only if their symbols are equal or transposes
of one another (Lemma~\ref{lem:clifford}). The restrictions of
the representations in $\tF$ form a single Lusztig family
\cite[Proposition~3.2 and the proof of Proposition~4.19]{Pan}.
We write
\[
\cF=\{\ResG\pi:\pi\in\tF\}.
\]
}Let $\sim$ be the equivalence relation
$(x^+,x^-)\sim(x^+,-x^-)$, write $[(x^+,x^-)]$ for the class of $(x^+,x^-)$,
and put
\begin{equation}\label{eq:SO-param}
\ResG\pi_{\cO^{\st}_{\tF},x^+,x^-}=\sigma_{\cF,[(x^+,x^-)]}.
\end{equation}
The resulting map is a bijection (Lemma~\ref{lem:clifford}):
\[
\mathcal M_\epsilon/{\sim}\xrightarrow{\ \sim\ }\cF,\qquad
[(x^+,x^-)]\longmapsto\sigma_{\cF,[(x^+,x^-)]}.
\]
The involution is fixed-point-free for $d\geqslant1$, so
$|\cF|=2^{2d-2}$.

{Unlike Lusztig's parametrization for the connected special
orthogonal group, which uses its canonical quotient \cite{Lus84}, we use
$\bA_{\mathrm O}(\cO^{\st}_{\tF})$.
The set $\mathcal M(\bA_{\mathrm O}(\cO^{\st}_{\tF}))$ simultaneously
parametrizes the unipotent representations of $\mathrm O_{2n}^{+}(\kk)$
and $\mathrm O_{2n}^{-}(\kk)$ whose symbols have the fixed multiset of entries;
the condition $\epsilon_d=\epsilon$ selects the form.
After identifying $(x^+,x^-)$ with $(x^+,-x^-)$, the two subsets
$\mathcal M_\epsilon/{\sim}$ parametrize the corresponding families for
$\mathrm{SO}_{2n}^{\epsilon}(\kk)$, whose unipotent supports have the same
partition.}

If the partition $\lambda$ of $\cO_{\pi}^{\st}$ for $\pi\in\tF$ is very even,
{the symbol $\Lambda_\pi$ attached to each $\pi\in\tF$ is degenerate}, $d=0$ and $\epsilon=+$. {The definition in \eqref{eq:fixed-form-parameters} applies
only when $d\geqslant1$. For $d=0$, the canonical quotient is trivial and
the unique symbol has defect zero, so the corresponding parameter sets are}
${\mathcal M_+=\{(0,1)\}}$ and $\mathcal M_-=\varnothing$. \finalrev{Thus $\tF=\{\pi\}$. The degenerate symbol is fixed by
transposition, so $\pi\otimes\det=\pi$, and
Lemma~\ref{lem:clifford} gives two distinct irreducible
constituents on restriction. Each forms a singleton Lusztig
family, denoted by $\cF^{\I}$ and $\cF^{\II}$; see
\cite[(4.6.10)]{Lus84} and the discussion following
\cite[Theorem~3.3]{LM}.} Here $\bA_{\mathrm{O}}(\cO^{\st}_{\tF})$ is trivial, and the second coordinate $1$ denotes its trivial
character. Thus
\[
\ResG\pi_{\cO^{\st}_{\tF},0,1}
=\sigma_{\cF^{\I},[(0,1)]}\oplus\sigma_{\cF^{\II},[(0,1)]}.
\]
Write these two representations as $\sigma^{\I}$ and $\sigma^{\II}$, so
that $\cF^{\I}=\{\sigma^{\I}\}$ and $\cF^{\II}=\{\sigma^{\II}\}$; with the
labels of Proposition~\ref{prop:separation},
$\sO(\sigma^{\I})^{\mx}=\cO_\lambda^{\I}$ and
$\sO(\sigma^{\II})^{\mx}=\cO_\lambda^{\II}$.

\smallskip
\noindent\emph{Odd orthogonal groups.}
Suppose $\dim V=2n+1$, so that
$G^F=\mathrm{O}_{2n+1}^{\epsilon}(\kk)$ with $\epsilon=\disc(V)$.
A family $\cF$ of $G_0^F$ is parametrized
by $\mathcal{M}(\bA(\cO^{\st}_{\cF}))$, where the canonical quotient is now the
one defined for $G_0$, as in \cite[Section~2.6]{Wang}. We write
$\sigma_{\cO^{\st}_{\cF},x^+,x^-}$ for the representation with parameter
$(x^+,x^-)$. Its two extensions to $G^F$ are
$\pi_{\cO^{\st}_{\cF},x^+,x^-,\eta}$, where $\eta\in\{\pm1\}$ is the scalar by
which $-I$ acts. Thus
\[
\ResG\pi_{\cO^{\st}_{\cF},x^+,x^-,\eta}=\sigma_{\cO^{\st}_{\cF},x^+,x^-},\qquad
\pi_{\cO^{\st}_{\cF},x^+,x^-,\eta}\otimes\det
=\pi_{\cO^{\st}_{\cF},x^+,x^-,-\eta}.
\]
In particular, twisting by $\det$ leaves $x^+$ unchanged in odd dimension
as well.

{When $\disc(V)=+$, we call $\sigma\in\cF$ the
\emph{special member} of $\cF$ if it is a constituent of the restriction
of a unipotent representation of $G^F$ attached to a special symbol of
defect zero in even dimension or defect one in odd dimension, as in
\cite[Section~2.6]{Wang}.}

\begin{lemma}\label{lem:unip-one}
Let $\sigma\in\Irr(G_0^F)$ be unipotent. For every $\cO_0\in\sO(\sigma)^{\mx}$,
we have $\dim\Wh_{\cO_0}(\sigma)=1$.
\end{lemma}

\begin{proof}
Choose a unipotent $\pi\in\Irr(G^F)$ whose restriction contains $\sigma$. By
Theorem~\ref{thm:max}, $\pr(\cO_0)\in\sO(\pi)^{\mx}$, and
$\dim\Wh_{\pr(\cO_0)}(\pi)=1$ by
\cite[Theorems~8.11(iv) and 8.12(iv)]{Wang}. Hence
$0<\dim\Wh_{\cO_0}(\sigma)\leqslant\dim\Wh_{\pr(\cO_0)}(\pi)=1$ by
Proposition~\ref{prop:restriction}.
\end{proof}

\subsection{{Rational orbit coordinates}}\label{subsec:coords}

\begingroup
Let $\cO^{\st}$ be a stable $G_0$-orbit with partition
$\lambda=[\lambda_1^{n_1},\ldots,\lambda_t^{n_t}]$, where
$\lambda_1>\cdots>\lambda_t$, and put $I_\lambda=\{i:\lambda_i\text{ is odd}\}$.
Choose $X\in(\cO^{\st})^F$ and a quadratic space $V^+$ with
\[
\dim V^+=\dim V,\qquad \disc(V^+)=+,\qquad
G^+=\mathrm O(V^+\otimes_{\kk}\kb).
\]
Choose $X_\lambda^+\in\Lie(G^+)^F$ of partition $\lambda$ and put
\[
\cO_\lambda^{+,\st}=\Ad(G^+)X_\lambda^+.
\]
For the component groups, we have
\[
\begin{aligned}
A_{\mathrm O}(\cO^{\st})&:=C_G(X)/C_G(X)^\circ
\simeq\bigoplus_{i\in I_\lambda}(\mathbb Z/2\mathbb Z)x_i,\\
A_{G^+}(\cO_\lambda^{+,\st})&:=C_{G^+}(X_\lambda^+)/C_{G^+}(X_\lambda^+)^\circ
\simeq\bigoplus_{i\in I_\lambda}(\mathbb Z/2\mathbb Z)x_i^+;
\end{aligned}
\]
see \cite[Section~3.1]{Wang}. Here $x_i$ and $x_i^+$ correspond to the odd
part $\lambda_i$. Let $\overline A_{G^+}(\cO_\lambda^{+,\st})$ be the canonical
quotient of $A_{G^+}(\cO_\lambda^{+,\st})$ and let
\[
q_{G^+}\colon A_{G^+}(\cO_\lambda^{+,\st})\twoheadrightarrow
\overline A_{G^+}(\cO_\lambda^{+,\st})
\]
be the quotient map; see \cite[Section~3.1]{Wang}.
We identify the component groups by
\[
\iota_\lambda\colon A_{\mathrm O}(\cO^{\st})\xrightarrow{\ \sim\ }A_{G^+}(\cO_\lambda^{+,\st}),
\qquad x_i\longmapsto x_i^+,
\]
and write
\[
A_{\mathrm O}=A_{\mathrm O}(\cO^{\st}),\qquad
\bA_{\mathrm O}^{+}:=\overline A_{G^+}(\cO_\lambda^{+,\st}),\qquad
q_{\mathrm O}=q_{G^+}\circ\iota_\lambda,
\]
Here $\bA_{\mathrm O}^{+}$ is the group denoted by $\bA_{\mathrm O}$
in Section~\ref{subsec:families}; the superscript $+$ records the choice of $V^+$.

Suppose first that $\lambda$ is not very even. Let
\[
\cS_{\mathrm O}\colon
\{\cO\in\Nil^F/\Ad(G^F):\cO\text{ has partition }\lambda\}
\longrightarrow A_{\mathrm O}
\]
be the map for $G^F=\mathrm O(V)$ defined by
\[
\cS_{\mathrm O}=\iota_\lambda^{-1}\circ S,
\]
where $S$ is the map in \cite[Section~3.1]{Wang} with values in
$A_{G^+}(\cO_\lambda^{+,\st})$.
Choose $X_\lambda^+$ in the reference rational $G^{+F}$-orbit used in
that definition. The map $\cS_{\mathrm O}$ depends on this rational orbit,
but not on its representative $X_\lambda^+$. It is also independent of
$X$ once the component groups are identified by $x_i\mapsto x_i^+$.
For a rational $G_0^F$-orbit
$\cO_0\subset(\cO^{\st})^F$, define
\[
\cS(\cO_0)=\cS_{\mathrm O}\bigl(\pr(\cO_0)\bigr),
\qquad
\bS(\cO_0)=q_{\mathrm O}\bigl(\cS(\cO_0)\bigr).
\]
By the orbit classification in Section~\ref{sec:orbits} and
Lemma~\ref{lem:split}, we have the following bijection:
\[
\cS\colon
\{\cO_0\in\Nil^F/\Ad(G_0^F):\cO_0\subset(\cO^{\st})^F\}
\xrightarrow{\ \sim\ }
T_V:=\{a\in A_{\mathrm O}:\det(a)=\disc(V)\}.
\]
Here $\det\colon A_{\mathrm O}\to\{\pm1\}$ is induced by the determinant
on $C_G(X)$; in particular, $\det(x_i)=-1$ for $i\in I_\lambda$.

For $y\in\bA_{\mathrm O}^{+}$, write
$y=(y_1,\ldots,y_r)$, with $y_j\in\{\pm1\}$, using the coordinates of
\cite[Section~3.1]{Wang}, where $r=\dim_{\mathbb F_2}\bA_{\mathrm O}^{+}$.
In these coordinates,
\[
\overline{\det}_{\mathrm O}\colon\bA_{\mathrm O}^{+}\longrightarrow\{\pm1\},
\qquad
\overline{\det}_{\mathrm O}(y_1,\ldots,y_r)=\prod_{j=1}^r y_j,
\qquad
\overline{\det}_{\mathrm O}\circ q_{\mathrm O}=\det.
\]
Since $q_{\mathrm O}$ is surjective, it follows that
\begin{equation}\label{eq:quotient-fixed-form}
q_{\mathrm O}(T_V)
=\{y\in\bA_{\mathrm O}^{+}:\overline{\det}_{\mathrm O}(y)=\disc(V)\}.
\end{equation}
For $x\in\bA_{\mathrm O}^{+}$, write
\[
\sO(x):=\bS^{-1}(x),
\]
where we regard the inverse image as the union of its rational
$G_0^F$-orbits in $(\cO^{\st})^F$.

If $\lambda$ is very even, then $A_{\mathrm O}=\bA_{\mathrm O}^{+}=1$.
We define $\cS$ and $\bS$ to be zero on the unique rational orbit
$(\cO^{\st})^F$, separately for each of the two stable $G_0$-orbits.
With this convention,
\[
\sO(0)=(\cO^{\st})^F,\qquad
\cS\bigl(\sO(0)\bigr)=\ker q_{\mathrm O}=\{0\}.
\]
\endgroup

\subsection{Canonical quotients and maximal rational orbits}

\begin{theorem}\label{thm:fibers}
Let $\sigma$ be an irreducible unipotent representation of $G_0^F$,
and put $\cO^{\st}=\cO_{\sigma}^{\st}$. Then there exists
$x\in\bA_{\mathrm{O}}(\cO^{\st})$ such that
\[
\sO(\sigma)^{\mx}=\sO(x).
\]
Moreover, $\sO(\sigma)^{\mx}$ is independent of $\psi$.
If $\disc(V)=+$ and $\sigma$ is the special member of its family,
in the sense of Section~\ref{subsec:families}, then
\[
\cS\bigl(\sO(\sigma)^{\mx}\bigr)=\ker q_{\mathrm{O}}.
\]
\end{theorem}

\begin{proof}
Choose a unipotent $\pi\in\Irr(G^F)$ whose restriction
contains $\sigma$, and let $\lambda$ be the partition
of $\cO_{\sigma}^{\st}$.
Suppose first that $\lambda$ is not very even.
By Theorem~\ref{thm:max}, the partition of
$\cO_{\pi}^{\st}$ is also $\lambda$.
In this case, $\ResG\pi=\sigma$, by
\eqref{eq:SO-param} in even dimension and by
Section~\ref{subsec:families} in odd dimension.
Hence by Theorem~\ref{thm:irred}, we have
\[
\sO(\sigma)^{\mx}
=\pr^{-1}\bigl(\sO(\pi)^{\mx}\bigr).
\]
\finalrev{We have $\sO(\sigma)^{\mx}=\sO(x)$ for some
$x\in\bA_{\mathrm O}(\cO^{\st})$. If $\disc(V)=+$ and
$\sigma$ is the special member of its family, then
$\cS\bigl(\sO(\sigma)^{\mx}\bigr)=\ker q_{\mathrm O}$
\cite[Theorem~1.10]{Wang}.
By \cite[Theorem~9.2]{Wang}, $\sO(\pi)^{\mx}$ is independent
of $\psi$. Since $\pr$ is independent of $\psi$, the same
holds for $\sO(\sigma)^{\mx}$.}

Now suppose that $\lambda$ is very even.
Then $\disc(V)=+$,
$A_{\mathrm{O}}(\cO^{\st})=\bA_{\mathrm{O}}(\cO^{\st})=1$,
and $\sigma$ is the special member of its
singleton family. By Proposition~\ref{prop:separation},
\[
\sO(\sigma)^{\mx}
=(\cO^{\st})^F=\sO(0),
\qquad
\cS\bigl(\sO(\sigma)^{\mx}\bigr)
=\ker q_{\mathrm{O}}=\{0\}.
\]
Finally, by \cite[Lemma~14.15]{Tay},
$\cO^{\st}$ is the unipotent support of
the irreducible Alvis--Curtis dual of $\sigma$,
and hence is independent of $\psi$.
Since $(\cO^{\st})^F$ is a single rational
$G_0^F$-orbit, $\sO(\sigma)^{\mx}$ is also
independent of $\psi$.
\end{proof}

\subsection{Alvis--Curtis duality}\label{subsec:AC}

Let $\sigma\in\cF$ be an irreducible unipotent representation of $G_0^F$, with
irreducible Alvis--Curtis dual $\sigma'$ as in the introduction. Choose a
unipotent $\pi\in\Irr(G^F)$ whose restriction contains $\sigma$, let $\tF$ be
its family, and let $\pi'$ be the irreducible Alvis--Curtis dual of $\pi$. For
unipotent representations, Alvis--Curtis duality is compatible with
restriction from $G^F$ to $G_0^F$ \cite{Lus82,Lus84}:
\[
\sigma\hookrightarrow\ResG\pi\implies\sigma'\hookrightarrow\ResG\pi'.
\]
Write $\pi=\pi_{\cO^{\st}_{\tF},x^+,x^-}$ in the even case and
$\pi=\pi_{\cO^{\st}_{\tF},x^+,x^-,\eta}$ in the odd case and write $x^+=(\epsilon_1,\dots,\epsilon_d)$,
$\epsilon_i\in\{\pm1\}$, as in
Section~\ref{subsec:families}. The
map of \cite[(9.52)]{Wang} is
\[
\Phi_{\tF}(x^+)=
\begin{cases}
(\epsilon_1,\ \epsilon_1\epsilon_2,\ \dots,\ \epsilon_{d-1}\epsilon_d),
&\dim V\text{ even},\\
(\disc(V)\epsilon_1,\ \epsilon_1\epsilon_2,\ \dots,\ \epsilon_{d-1}\epsilon_d,\ \epsilon_d),
&\dim V\text{ odd}.
\end{cases}
\]
The coordinates of $x^+$ are the family coordinates of
Section~\ref{subsec:families}. The coordinates of
$\Phi_{\tF}(x^+)$ are the Sommers coordinates of
Section~\ref{subsec:coords}, regarded as coordinates on
$\bA_{\mathrm{O}}(\cO^{\st}_{\cF})$.
When $d=0$, it is represented by the empty tuple in the
even case and by $(\disc(V))$ in the odd case.

\rev{We verify that $\Phi_{\tF}(x^+)$ is compatible with $V$.}
In even dimension with
$d\geqslant1$, \eqref{eq:fixed-form-parameters} gives
\[
\overline{\det}_{\mathrm O}\bigl(\Phi_{\tF}(x^+)\bigr)
=\epsilon_1\prod_{i=2}^d(\epsilon_{i-1}\epsilon_i)
=\epsilon_d=\disc(V).
\]
In odd dimension, the additional final coordinate $\epsilon_d$ makes the
product equal to $\disc(V)$; for $d=0$ this follows from the stated
one-coordinate convention.

\begin{theorem}\label{thm:AC}
Let $\sigma\in\cF$ be an irreducible unipotent
representation of $G_0^F$, and let $\sigma'$ be its
irreducible Alvis--Curtis dual. Choose a unipotent
$\pi\in\Irr(G^F)$ whose restriction contains $\sigma$,
and let $\tF$ and $(x^+,x^-)$ be the family and
parameter of $\pi$, respectively. Then
\[
\sO(\sigma')^{\mx}
=\sO\bigl(\Phi_{\tF}(x^+)\bigr).
\]
The right-hand side is independent of the choice of $\pi$.
\end{theorem}

\begin{proof}
Let $\pi'$ be the irreducible Alvis--Curtis dual
of $\pi$. By \cite[Lemma~14.15]{Tay},
$\cO_{\sigma'}^{\st}=\cO^{\st}_{\cF}$.

Suppose first that the partition of
$\cO^{\st}_{\cF}$ is not very even.
As in the proof of Theorem~\ref{thm:fibers},
we have
\[
\sO(\sigma')^{\mx}
=\pr^{-1}\bigl(\sO(\pi')^{\mx}\bigr).
\]
The desired equality then follows from
\cite[Theorem~1.11]{Wang},
$\cS=\cS_{\mathrm{O}}\circ\pr$, and
\eqref{eq:quotient-fixed-form}.
By Lemma~\ref{lem:clifford}, any other irreducible
lift of $\sigma$ is $\pi\otimes\det$.
As shown in Section~\ref{subsec:families},
twisting by $\det$ preserves $x^+$ in even
dimension and changes only $\eta$ in odd
dimension. Thus the right-hand side is
independent of the choice of $\pi$.

When the partition of $\cO^{\st}_{\cF}$ is
very even, the desired equality follows
immediately from Theorem~\ref{thm:max},
Lemma~\ref{lem:split}, and the triviality of
$\bA_{\mathrm{O}}(\cO^{\st}_{\cF})$.
In this case, independence of the choice of
$\pi$ is also immediate.
\end{proof}
% ===============================================================
\section*{\texorpdfstring{\rev{Acknowledgments}}{Acknowledgments}}

Z.~Wang gratefully acknowledges financial support from the National
Natural Science Foundation of China (Grant No.~12571012) and the Jilin
Provincial Postdoctoral Foundation. P.~Li acknowledges support from the New
Cornerstone Science Foundation through the New Cornerstone Investigator Program
awarded to Professor Xuhua He.

\section*{{Declaration of AI use}}

\rev{We used Claude (Anthropic) to generate Figures~\ref{fig:O8plus}
and~\ref{fig:O8minus}, and GPT-6 (OpenAI) to check and polish the manuscript.
All statements and proofs have been checked by the authors, who take full
responsibility for the content of this paper.}

% ===============================================================

\end{document}